\documentclass[12pt,reqno]{amsart}%
\usepackage{amsmath, amsfonts, amssymb, amsthm, amscd, amsbsy}
\usepackage[dvipsnames,svgnames,x11names,hyperref]{xcolor}
\usepackage{geometry}
\usepackage{graphicx}
\usepackage{hyperref}
\usepackage{lineno}%
\usepackage{amsmath}%
\usepackage{amsfonts}%
\usepackage{amssymb}
\providecommand{\U}[1]{\protect\rule{.1in}{.1in}}
\hypersetup{colorlinks,breaklinks,
linkcolor={Fuchsia},
citecolor={ForestGreen},
urlcolor={NavyBlue}}
\newtheorem{theorem}{Theorem}[section]
\theoremstyle{plain}

\newtheorem{corollary}{Corollary}[section]

\numberwithin{equation}{section}
\allowdisplaybreaks
\begin{document}
\title[Hardy identities and inequalities]{Hardy's identities and inequalities on Cartan-Hadamard manifolds}
\author{Joshua Flynn}
\address{Joshua Flynn: Department of Mathematics\\
University of Connecticut\\
Storrs, CT 06269, USA}
\email{joshua.flynn@uconn.edu}
\author{Nguyen Lam}
\address{Nguyen Lam: School of Science \& Environment\\
Grenfell Campus, Memorial University of Newfoundland\\
Corner Brook, NL A2H5G4, Canada }
\email{nlam@grenfell.mun.ca}
\author{Guozhen Lu }
\address{Guozhen Lu: Department of Mathematics\\
University of Connecticut\\
Storrs, CT 06269, USA}
\email{guozhen.lu@uconn.edu}
\author{Saikat Mazumdar}
\address{Saikat Mazumdar: Department of Mathematics\\
Indian Institute of Technology Bombay\\
Mumbai 400076, India}
\email{saikat@math.iitb.ac.in, saikat.mazumdar@iitb.ac.in}

\begin{abstract}
We study the Hardy identities and inequalities on Cartan-Hadamard manifolds
using the notion of a Bessel pair. These Hardy identities offer significantly
more information on the existence/nonexistence of the extremal functions of
the Hardy inequalities. These Hardy inequalities are in the spirit of
Brezis-V\'{a}zquez in the Euclidean spaces. As direct consequences, we
establish several Hardy type inequalities that provide substantial
improvements as well as simple understandings to many known Hardy inequalities
and Hardy-Poincar\'{e}-Sobolev type inequalities on hyperbolic spaces in the literature.

\end{abstract}
\maketitle

\section{Introduction}

The main purpose of this article is to study the improvements of the $L^{2}%
$-Hardy type inequalities on Cartan-Hadamard manifold, i.e. a Riemannian
manifold that is complete and simply connected and has everywhere nonpositive
sectional curvature. We also sharpen Hardy inequalities on hyperbolic spaces
in the literature.

We recall that on the Euclidean space $\mathbb{R}^{N}$, $N\geq3$, the
following celebrated Hardy inequality plays important roles in many areas such
as analysis, probability and partial differential equations:
\begin{equation}%
{\displaystyle\int\limits_{\mathbb{R}^{N}}}
\left\vert \nabla f\right\vert ^{2}dx\geq\left(  \frac{N-2}{2}\right)  ^{2}%
{\displaystyle\int\limits_{\mathbb{R}^{N}}}
\frac{\left\vert f\right\vert ^{2}}{\left\vert x\right\vert ^{2}}dx\text{,
}f\in C_{0}^{\infty}\left(
\mathbb{R}
^{N}\right)  \text{.} \label{1.1}%
\end{equation}
The constant $\left(  \frac{N-2}{2}\right)  ^{2}$ in the above Hardy
inequality is optimal and is never achieved by nontrivial functions.
Therefore, one may want to improve (\ref{1.1}) by adding extra nonnegative
terms to its right hand side. On the whole space $\mathbb{R}^{N}$, the
operator $-\Delta-\left(  \frac{N-2}{2}\right)  ^{2}\frac{1}{\left\vert
x\right\vert ^{2}}$ is known to be critical and there is no strictly positive
$V\in V^{1}\left(  \left(  0,\infty\right)  \right)  $ such that the
inequality%
\[%
{\displaystyle\int\limits_{\mathbb{R}^{N}}}
\left\vert \nabla f\right\vert ^{2}dx-\left(  \frac{N-2}{2}\right)  ^{2}%
{\displaystyle\int\limits_{\mathbb{R}^{N}}}
\frac{\left\vert f\right\vert ^{2}}{\left\vert x\right\vert ^{2}}dx\geq%
{\displaystyle\int\limits_{\mathbb{R}^{N}}}
V\left(  \left\vert x\right\vert \right)  \left\vert f\right\vert ^{2}dx
\]
holds for all $f\in C_{0}^{\infty}\left(
\mathbb{R}
^{N}\right)  $ (see \cite[Corollary 2.3.4]{GM1}, for e.g.). The situation is
very different on bounded domains. In particular, it has been showed that
extra nonnegative terms can be added to the Hardy inequality on bounded
domains. For instance, let $\Omega$ be a bounded domain in $%
\mathbb{R}
^{N}$, $N\geq3$, with $0\in\Omega$, then in order to investigate the stability
of singular solutions of nonlinear elliptic equations, Brezis and V\'{a}zquez
verified in \cite{BV} that for all $f\in W_{0}^{1,2}\left(  \Omega\right)  :$%
\begin{equation}%
{\displaystyle\int\limits_{\Omega}}
\left\vert \nabla f\right\vert ^{2}dx-\left(  \frac{N-2}{2}\right)  ^{2}%
{\displaystyle\int\limits_{\Omega}}
\frac{\left\vert f\right\vert ^{2}}{\left\vert x\right\vert ^{2}}dx\geq
z_{0}^{2}\omega_{N}^{\frac{2}{N}}\left\vert \Omega\right\vert ^{-\frac{2}{N}}%
{\displaystyle\int\limits_{\Omega}}
\left\vert f\right\vert ^{2}dx \label{1.3}%
\end{equation}
where $\omega_{N}$ is the volume of the unit ball and $z_{0}=2.4048...$ is the
first zero of the Bessel function $J_{0}\left(  z\right)  $. We also mention
that in \cite{VZ00}, V\'{a}zquez and Zuazua established the following improved
Hardy-Poincar\'{e} inequality: for any $1\leq q<2$, there exists a constant
$C\left(  q,\Omega\right)  >0$ such that for all $f\in W_{0}^{1,2}\left(
\Omega\right)  :$
\[%
{\displaystyle\int\limits_{\Omega}}
\left\vert \nabla f\right\vert ^{2}dx-\left(  \frac{N-2}{2}\right)  ^{2}%
{\displaystyle\int\limits_{\Omega}}
\frac{\left\vert f\right\vert ^{2}}{\left\vert x\right\vert ^{2}}dx\geq
C\left(  q,\Omega\right)  \left\Vert f\right\Vert _{W^{1,q}\left(
\Omega\right)  }^{2}.
\]

It is interesting to note that the constant $z_{0}^{2}\omega_{N}^{\frac{2}{N}%
}\left\vert \Omega\right\vert ^{-\frac{2}{N}}$ in (\ref{1.3}) is optimal when
$\Omega$ is a ball and again is not attained in $W_{0}^{1,2}\left(
\Omega\right)  $. Therefore, Brezis and V\'{a}zquez also conjectured that
$z_{0}^{2}\omega_{N}^{\frac{2}{N}}\left\vert \Omega\right\vert ^{-\frac{2}{N}}%
{\displaystyle\int\limits_{\Omega}}
\left\vert f\right\vert ^{2}dx$ is just the first term of an infinite series
of extra terms that can be added to the right hand side of (\ref{1.3}). This
problem has attracted great attention and was investigated by many authors.
See \cite{ACR, Beckner12, BM97, BMS00, D, DFP14, FS, GGM, GL2017}, among
others. We also refer the interested reader to \cite{BEL, GM1, KMP2007, KP,
Maz11, OK} which are excellent monographs on the topic. In particular, we note
that in an attempt to improve, extend and unify several results in this
direction, Ghoussoub and Moradifam \cite{GM} introduced the notion of a Bessel
pair and studied its connections to Hardy inequalities. One of their results
can be read as follows:

\vskip0.5cm

\textbf{Theorem A.} \textit{Let }$0<R\leq\infty$\textit{, }$B_{R}%
=B(0,R)$\textit{ be a ball centered at the origin with radius }$R$\textit{,
}$V$\textit{ and }$W$\textit{ be positive }$C^{1}$-\textit{functions on
}$\left(  0,R\right)  $\textit{ such that }$%
{\displaystyle\int\limits_{0}^{R}}
\frac{1}{r^{N-1}V\left(  r\right)  }dr=\infty$\textit{ and }$%
{\displaystyle\int\limits_{0}^{R}}
r^{N-1}V\left(  r\right)  dr<\infty$\textit{. Then,}

\textit{(1) If }$\left(  r^{N-1}V,r^{N-1}W\right)  $\textit{ is a Bessel pair
on }$\left(  0,R\right)  $\textit{, then for all}$\ f\in C_{0}^{\infty}\left(
B_{R}\right)  :$\textit{ }%
\begin{equation}%
{\displaystyle\int\limits_{B_{R}}}
V\left(  \left\vert x\right\vert \right)  \left\vert \nabla f\right\vert
^{2}dx\geq%
{\displaystyle\int\limits_{B_{R}}}
W\left(  \left\vert x\right\vert \right)  \left\vert f\right\vert
^{2}dx\text{\textit{.}} \label{A1}%
\end{equation}

\textit{(2) If (\ref{A1}) holds for all}$\ f\in C_{0}^{\infty}\left(
B_{R}\right)  $\textit{, then }$\left(  r^{N-1}V,r^{N-1}cW\right)  $\textit{
is a Bessel pair on }$\left(  0,R\right)  $\textit{ for some }$c>0.$

\vskip0.5cm

Here we say that a couple of $C^{1}$-functions $\left(  V,W\right)  $ is a
Bessel pair on $\left(  0,R\right)  $ for some $0<R\leq\infty$ if the ordinary
differential equation
\begin{equation}
\left(  Vy^{\prime}\right)  ^{\prime}+Wy=0 \label{1}%
\end{equation}
has a positive solution $\varphi$ on the interval $\left(  0,R\right)  $.

We also mention the paper \cite{Muck72} where Muckenhought pairs have been
used to study the necessary and sufficient conditions for the validity of the
Hardy inequality on one-dimensional space.

\vskip0.5cm

Hardy type inequalities have also been generalized to the cases with general
distance functions \cite{BFT, BM97, LLZ19, LLZ20}, in multipolar setting
\cite{BDE08, Caz16, CZ13}, etc. In particular, the following Hardy inequality
has been first established on Riemannian manifold $\left(  \mathbb{M}%
,g\right)  $ by Carron in the paper \cite{Car}:
\begin{equation}%
{\displaystyle\int\limits_{\mathbb{M}}}
\rho^{\alpha}\left(  x\right)  \left\vert \nabla_{g}f\right\vert _{g}%
^{2}dV_{g}\geq\left(  \frac{C+\alpha-1}{2}\right)  ^{2}%
{\displaystyle\int\limits_{\mathbb{M}}}
\rho^{\alpha}\left(  x\right)  \frac{\left\vert f\right\vert ^{2}}{\rho
^{2}\left(  x\right)  }dV_{g} \label{1.2}%
\end{equation}
where $\alpha\in%
\mathbb{R}
$, $C+\alpha-1>0$, $f\in C_{0}^{\infty}\left(  \mathbb{M}\setminus\rho
^{-1}\left\{  0\right\}  \right)  $ and the weighted function $\rho$ satisfies
the eikonal equation $\left\vert \nabla_{g}\rho\right\vert _{g}=1$ and
$\Delta_{g}\rho\geq\frac{C}{\rho}$ for some $C>0$. Here $dV_{g}$, $\nabla_{g}%
$, $\Delta_{g}$ and $\left\vert \cdot\right\vert _{g}$ denote the volume
element, gradient, Laplace--Beltrami operator and the length of a vector field
with respect to the Riemannian metric $g$ on $\mathbb{M}$, respectively.
Further developments have been established in \cite{BDGG17, BGG, DD14, KS19},
for instance.

When $\mathbb{M}$ is a $N$-dimensional Cartan--Hadamard manifold and
$\rho=d(x,O)$ is the geodesic distance, then $\rho$ satisfies all the
aforementioned conditions. Moreover, it was showed in \cite{Car} that
\begin{equation}%
{\displaystyle\int\limits_{\mathbb{M}}}
\left\vert \nabla_{g}f\right\vert _{g}^{2}dV_{g}\geq\left(  \frac{N-2}%
{2}\right)  ^{2}%
{\displaystyle\int\limits_{\mathbb{M}}}
\frac{\left\vert f\right\vert ^{2}}{\rho^{2}\left(  x\right)  }dV_{g}%
.\label{1.4}%
\end{equation}
Moreover, the constant $\left(  \frac{N-2}{2}\right)  ^{2}$ was verified to be
optimal in \cite{YSK}. In particular, when $\mathbb{M}$ is the hyperbolic
space $\mathbb{H}^{N}$, we have
\begin{equation}%
{\displaystyle\int\limits_{\mathbb{H}^{N}}}
\left\vert \nabla_{\mathbb{H}}f\right\vert ^{2}dV_{\mathbb{H}}\geq\left(
\frac{N-2}{2}\right)  ^{2}%
{\displaystyle\int\limits_{\mathbb{H}^{N}}}
\frac{\left\vert f\right\vert ^{2}}{\rho^{2}\left(  x\right)  }dV_{\mathbb{H}%
}\label{Ha1}%
\end{equation}
where $\rho\left(  x\right)  $ is the geodesic distance on $\mathbb{H}^{N}$.
On the other hand, it is well-known that on $\mathbb{H}^{N}$, the $L^{2}%
$-spectrum is $\left[  \left(  \frac{N-1}{2}\right)  ^{2},\infty\right)  $.
More precisely, we have the Poincar\'{e}--Sobolev inequality%
\begin{equation}%
{\displaystyle\int\limits_{\mathbb{H}^{N}}}
\left\vert \nabla_{\mathbb{H}}f\right\vert ^{2}dV_{\mathbb{H}}\geq\left(
\frac{N-1}{2}\right)  ^{2}%
{\displaystyle\int\limits_{\mathbb{H}^{N}}}
\left\vert f\right\vert ^{2}dV_{\mathbb{H}}\label{Hy1}%
\end{equation}
where $\left(  \frac{N-1}{2}\right)  ^{2}$ is sharp and is never attained by
nontrivial functions in $W^{1,2}\left(  \mathbb{H}^{N}\right)  $. In
\cite{AK13}, the authors investigated the finiteness and infiniteness of the
discrete spectrum of the Schr\"{o}dinger operator $-\Delta_{\mathbb{H}}+V$ and
set up the following sharp improvements of the Poincar\'{e}--Sobolev
inequality (\ref{Hy1}):%
\begin{align}
&
{\displaystyle\int\limits_{\mathbb{H}^{N}}}
\left\vert \nabla_{\mathbb{H}}f\right\vert ^{2}dV_{\mathbb{H}}-\left(
\frac{N-1}{2}\right)  ^{2}%
{\displaystyle\int\limits_{\mathbb{H}^{N}}}
\left\vert f\right\vert ^{2}dV_{\mathbb{H}}\nonumber\\
&  \geq\frac{1}{4}%
{\displaystyle\int\limits_{\mathbb{H}^{N}}}
\frac{\left\vert f\right\vert ^{2}}{\rho^{2}\left(  x\right)  }dV_{\mathbb{H}%
}+\frac{\left(  N-1\right)  \left(  N-3\right)  }{4}%
{\displaystyle\int\limits_{\mathbb{H}^{N}}}
\frac{\left\vert f\right\vert ^{2}}{\sinh^{2}\rho\left(  x\right)
}dV_{\mathbb{H}}.\label{Hy2}%
\end{align}
Moreover, the operator $-\Delta_{\mathbb{H}}-\left(  \frac{N-1}{2}\right)
^{2}-\frac{1}{4}\frac{1}{\rho^{2}\left(  x\right)  }-\frac{\left(  N-1\right)
\left(  N-3\right)  }{4}\frac{1}{\sinh^{2}\rho\left(  x\right)  }$ is critical
in $\mathbb{H}^{N}\setminus\left\{  0\right\}  $ in the sense that for any
$W>\frac{1}{4r^{2}}+\frac{\left(  N-1\right)  \left(  N-3\right)  }{4}\frac
{1}{\sinh^{2}r}$, the inequality%
\begin{align*}
&
{\displaystyle\int\limits_{\mathbb{H}^{N}}}
\left\vert \nabla_{\mathbb{H}}f\right\vert ^{2}dV_{\mathbb{H}}-\left(
\frac{N-1}{2}\right)  ^{2}%
{\displaystyle\int\limits_{\mathbb{H}^{N}}}
\left\vert f\right\vert ^{2}dV_{\mathbb{H}}\\
&  \geq%
{\displaystyle\int\limits_{\mathbb{H}^{N}}}
W\left\vert f\right\vert ^{2}dV_{\mathbb{H}}\text{ }\forall f\in C_{0}%
^{\infty}\left(  \mathbb{H}^{N}\setminus\left\{  0\right\}  \right)
\end{align*}
is not valid. This Hardy-Poincar\'{e}-Sobolev inequality has also been studied
on larger classes of manifolds in \cite{BGG}. Recently, there has been
progress of establishing higher order Hardy-Sobolev-Maz'ya inequalities on
hyperbolic spaces using Fourier analysis on hyperbolic spaces (see Lu and Yang
\cite{LuYang1, LuYang2}). It is also worth mentioning that the problems of
improving Hardy type inequalities as well as other functional and geometric
inequalities using the effect of curvature have been studied intensively
recently. We refer the interested reader to \cite{BGGP, CGMSO18, DH02, KO1,
KO2, NN19, Nguyen19, YSK}, to name just a few.

\vskip0.5cm

Motivated by the aforementioned results, the main purpose of this article is
to study the general Hardy type inequalities on Cartan-Hadamard manifolds.
Moreover, we will set up some general Hardy identities that can be used to
derive several substantial improvements of the Hardy inequality on
Cartan--Hadamard manifolds. Our equalities not only provide straightforward
understandings of several Hardy type inequalities, but also explain the
existence and nonexistence of nontrivial optimizers.

Let $(\mathbb{M},g)$ be a complete Riemannian manifold of dimension $N$. In a
local coordinate system $\left\{  x^{i}\right\}  _{i=1}^{N}$, we can write
\[
g=%
{\displaystyle\sum}
g_{ij}dx^{i}dx^{j}.
\]
The Laplace-Beltrami operator $\Delta_{g}$\ with respect to the metric $g$ may
then be written as
\[
\Delta_{g}=%
{\displaystyle\sum}
\frac{1}{\sqrt{\det\left(  g_{ij}\right)  }}\frac{\partial}{\partial x^{i}%
}\left(  \sqrt{\det\left(  g_{ij}\right)  }g^{ij}\frac{\partial}{\partial
x^{j}}\right)
\]
where $\left(  g^{ij}\right)  =\left(  g_{ij}\right)  ^{-1}$. Denote by
$\nabla_{g}$ the corresponding gradient. Then
\[
\left\langle \nabla_{g}f,\nabla_{g}h\right\rangle _{g}=%
{\displaystyle\sum}
g^{ij}\frac{\partial f}{\partial x^{i}}\frac{\partial g}{\partial x^{j}}.
\]
We also denote%
\[
\left\vert \nabla_{g}f\right\vert _{g}=\sqrt{\left\langle \nabla_{g}%
f,\nabla_{g}f\right\rangle _{g}}.
\]
Fix a point $O\in\mathbb{M}$ and denote by $\rho(x)=d(x,O)$ for all
$x\in\mathbb{M}$, where $d$ denotes the geodesic distance on $\mathbb{M}$.
Then $\rho\left(  x\right)  $ is Lipschitz continuous in $\mathbb{M}$.

For each point $O\in\mathbb{M}$, consider the exponential map $\exp_{O}%
:T_{O}\mathbb{M}\rightarrow\mathbb{M}$. For $X \in T_{O}\mathbb{M}$, let
$\gamma(t)$ be the unique geodesic such that $\gamma(0)=O$ and $\gamma
^{\prime}(0)=X$. Then $\exp_{O}(tX)=\gamma(t)$ for $t>0$. For small $t$,
$\gamma$ is the unique minimal geodesic joining the points $O$ and $\exp
_{O}(tX)$.

One can write
\begin{align*}
\mathbb{M}=\exp_{O}(U_{O})\cup Cut(O),
\end{align*}
where $Cut(O)$ denotes the cut locus of the point $O$ and $U_{O}$ is an open
neighborhood of $O$ in $T_{O}\mathbb{M}$. Furthermore, $\exp_{O}:
U_{O}\rightarrow\exp_{O}(U_{O})$ is a diffeomorphism and $Cut(O)= \exp
_{O}\partial U_{O}$. Also, the cut locus $Cut(O)$ has measure zero .

The distance function $\rho(x)$ is smooth on $\mathbb{M}\setminus
\big( Cut(O)\cup\{O\} \big)$ and it satisfies $\left\vert \nabla_{g}%
\rho\left(  x\right)  \right\vert _{g}=1$ on $\mathbb{M}\setminus
\big( Cut(O)\cup\{O\} \big)$.

For a Cartan-Hadamard manifold $(M,g)$, the exponential map $\exp_{O}%
:T_{O}\mathbb{M}\rightarrow\mathbb{M}$ is a diffeomorphism and
$Cut(O)=\emptyset$, and then $\rho\left(  x\right)  $ is smooth in
$\mathbb{M}\setminus\left\{  O\right\}  $ and $\left\vert \nabla_{g}%
\rho\left(  x\right)  \right\vert _{g}=1$.

For any $R>0$, denote by $B_{R}\left(  O\right)  =\{x\in\mathbb{M}%
:\rho(x)<\delta\}$ the geodesic ball in $\mathbb{M}$ with center at $O$ and
radius $R$. Now, we choose an orthonormal basis $\left\{  u,e_{2}%
,...,e_{N}\right\}  $ in $T_{O}\mathbb{M}$ and let $c\left(  t\right)
=\exp_{O}\left(  tu\right)  $ be a geodesic curve. Consider the Jacobi fields
$\left\{  Y_{2}\left(  t\right)  ,...,Y_{N}\left(  t\right)  \right\}  $
satisfying $Y_{i}\left(  0\right)  =0$ and $Y_{i}^{\prime}\left(  0\right)
=e_{i}$, so that the volume density function written in geodesic polar
coordinates can be given by%
\[
J\left(  u,t\right)  =t^{-N+1}\sqrt{\det\left(  \left\langle Y_{i}\left(
t\right)  ,Y_{j}\left(  t\right)  \right\rangle \right)  },\text{ }t>0\text{.}%
\]
We note that $J\left(  u,t\right)  \in C^{\infty}\left(  T_{O}\mathbb{M}%
\setminus\left\{  O\right\}  \right)  $ and does not depend on $\left\{
e_{2},...,e_{N}\right\}  $. By the definition of the density function
$J\left(  u,t\right)  $, we have the polar coordinates on $\mathbb{M}$:%
\[%
{\displaystyle\int\limits_{\mathbb{M}}}
f\left(  x\right)  dV_{g}=%
{\displaystyle\int\limits_{\mathbb{S}^{N-1}}}
{\displaystyle\int\limits_{0}^{\infty}}
f\left(  \exp_{O}\left(  tu\right)  \right)  J\left(  u,t\right)
t^{N-1}dtdu.
\]
Here $du$ denotes the canonical measure of the unit sphere of $T_{O}%
\mathbb{M}$.

For any function $f$ on $\mathbb{M}$, we also define the radial derivation
$\partial_{\rho}=\frac{\partial}{\partial\rho}$ along the geodesic curve
starting from $O$ by
\[
\partial_{\rho}f\left(  x\right)  =\frac{d\left(  f\circ\exp_{O}\right)  }%
{dr}\left(  \exp_{O}^{-1}\left(  x\right)  \right)  .
\]
Here we denote $\frac{d}{dr}$ the radial derivation on $T_{O}\mathbb{M}$:%
\[
\frac{d}{dr}F\left(  u\right)  =\left\langle \frac{u}{\left\vert u\right\vert
},\nabla F(u)\right\rangle .
\]
We note that by Gauss's lemma, we have that $\left\vert \partial_{\rho
}f\right\vert \leq\left\vert \nabla_{g}f\right\vert _{g}$ for $f\in
C^{1}\left(  \mathbb{M}\setminus Cut(O)\right)  $.

The first main result of this article is the following Hardy type identities
on the general complete Riemannian manifold $(\mathbb{M},g):$

\begin{theorem}
\label{T1} Let $(\mathbb{M},g)$ be a complete Riemannian manifold of dimension
$N$. \textit{Let } $O\in\mathbb{M}$ and take $0<R\leq d(O,Cut(O))$.
\textit{Let} $V$\textit{and }$W$\textit{ be positive }$C^{1}-$%
\textit{functions on }$\left(  0,R\right)  $ such that $\left(  r^{N-1}%
V,r^{N-1}W\right)  $ is a Bessel pair on $\left(  0,R\right)  $. Then we have
the following identities for all $f\in C_{0}^{\infty}\left(  \mathbb{M}%
\setminus\big(Cut(O)\cup\rho^{-1}\left\{  0\right\}  \big)\right)  $:%
\begin{align*}
&
{\displaystyle\int\limits_{B_{R}\left(  O\right)  }}
V\left(  \rho\left(  x\right)  \right)  \left\vert \nabla_{g}f\right\vert
_{g}^{2}dV_{g}-%
{\displaystyle\int\limits_{B_{R}\left(  O\right)  }}
W\left(  \rho\left(  x\right)  \right)  \left\vert f\right\vert ^{2}dV_{g}\\
&  =%
{\displaystyle\int\limits_{B_{R}\left(  O\right)  }}
V\left(  \rho\left(  x\right)  \right)  \left\vert \varphi^{2}\left(
\rho\left(  x\right)  \right)  \right\vert \left\vert \nabla_{g}\left(
\frac{f}{\varphi\left(  \rho\left(  x\right)  \right)  }\right)  \right\vert
_{g}^{2}dV_{g}\\
&  -%
{\displaystyle\int\limits_{B_{R}\left(  O\right)  }}
V\left(  \rho\left(  x\right)  \right)  \left\vert f\right\vert ^{2}%
\frac{\varphi^{\prime}\left(  \rho\left(  x\right)  \right)  }{\varphi\left(
\rho\left(  x\right)  \right)  }\frac{J^{\prime}\left(  u,\rho\left(
x\right)  \right)  }{J\left(  u,\rho\left(  x\right)  \right)  }dV_{g}%
\end{align*}
and%
\begin{align*}
&
{\displaystyle\int\limits_{B_{R}\left(  O\right)  }}
V\left(  \rho\left(  x\right)  \right)  \left\vert \partial_{\rho}f\right\vert
^{2}dV_{g}-%
{\displaystyle\int\limits_{B_{R}\left(  O\right)  }}
W\left(  \rho\left(  x\right)  \right)  \left\vert f\right\vert ^{2}dV_{g}\\
&  =%
{\displaystyle\int\limits_{B_{R}\left(  O\right)  }}
V\left(  \rho\left(  x\right)  \right)  \varphi^{2}\left(  \rho\left(
x\right)  \right)  \left\vert \partial_{\rho}\left(  \frac{f}{\varphi\left(
\rho\left(  x\right)  \right)  }\right)  \right\vert ^{2}dV_{g}\\
&  -%
{\displaystyle\int\limits_{B_{R}\left(  O\right)  }}
V\left(  \rho\left(  x\right)  \right)  \left\vert f\right\vert ^{2}%
\frac{\varphi^{\prime}\left(  \rho\left(  x\right)  \right)  }{\varphi\left(
\rho\left(  x\right)  \right)  }\frac{J^{\prime}\left(  u,\rho\left(
x\right)  \right)  }{J\left(  u,\rho\left(  x\right)  \right)  }dV_{g}.
\end{align*}
Here $J^{\prime}\left(  u,t\right)  =\frac{\partial J\left(  u,t\right)
}{\partial t}$, $x=\exp_{O}\left(  \rho u\right)  $ and $\varphi$ is the
positive solution of%
\[
\left(  r^{N-1}V\left(  r\right)  \varphi^{\prime}(r)\right)  ^{\prime
}+r^{N-1}W\left(  r\right)  \varphi(r)=0.
\]

\end{theorem}

\vskip0.5cm

It is in fact possible to consider $\varphi$ which have a zero at some $r=R$,
but are positive elsewhere and which satisfy the Bessel pair ODE on $\left(
0,R\right)  \cup\left(  R,\infty\right)  $. As the following theorem
demonstrates, considering such $\varphi$ allows one to establish global Hardy
identities provided $f$ is replaced by $f-f\left(  \exp\right)  $ and provided
$f$ satisfies (\ref{C}). We note that on finite interval $\left(  0,R\right)
$, this function $\varphi$ satisfies the condition in Theorem \ref{T1}.
However, on the infinite interval $\left(  0,\infty\right)  $, this $\varphi$
is allowed to be degenerate or singular at $R$.

\begin{theorem}
\label{T2} Let $(\mathbb{M},g)$ be a complete Riemannian manifold of dimension
$N$. \textit{Let } $O\in\mathbb{M}$ and take $0<R\leq d(O,Cut(O))$ Assume that
$V$ and $W$ are positive $C^{1}$-functions on $\left(  0,R\right)  \cup\left(
R,\infty\right)  $ such that the ordinary differential equation
\[
\left(  V\left(  r\right)  r^{N-1}\varphi^{\prime}\left(  r\right)  \right)
^{\prime}+W\left(  r\right)  r^{N-1}\varphi\left(  r\right)  =0
\]
has a positive solution $\varphi$ on $\left(  0,R\right)  \cup\left(
R,\infty\right)  $.

\noindent Then for all $f\in C_{0}^{\infty}\left(  \mathbb{M}\setminus
\big( Cut(O)\cup\rho^{-1}\left\{  0\right\}  \big) \right)  $ satisfying that
for all $u\in\mathbb{S}^{N-1}:$%
\begin{equation}
\lim_{r\rightarrow R}V\left(  r\right)  \frac{\varphi^{\prime}\left(
r\right)  }{\varphi\left(  r\right)  }\left\vert f\left(  \exp_{O}\left(
ru\right)  \right)  -f\left(  \exp_{O}\left(  Ru\right)  \right)  \right\vert
^{2}=0, \label{C}%
\end{equation}
we have
\begin{align*}
&
{\displaystyle\int\limits_{\mathbb{M}}}
V\left(  \rho\left(  x\right)  \right)  \left\vert \nabla_{g}\left(
f-f\left(  \exp_{O}\left(  Ru\right)  \right)  \right)  \right\vert _{g}%
^{2}dx-%
{\displaystyle\int\limits_{\mathbb{M}}}
W\left(  \rho\left(  x\right)  \right)  \left\vert f-f\left(  \exp_{O}\left(
Ru\right)  \right)  \right\vert ^{2}dV_{g}\\
&  =%
{\displaystyle\int\limits_{\mathbb{M}}}
V\left(  \rho\left(  x\right)  \right)  \varphi^{2}\left(  \rho\left(
x\right)  \right)  \left\vert \nabla_{g}\left(  \frac{f-f\left(  \exp
_{O}\left(  Ru\right)  \right)  }{\varphi\left(  \rho\left(  x\right)
\right)  }\right)  \right\vert _{g}^{2}dV_{g}\\
&  -%
{\displaystyle\int\limits_{\mathbb{M}}}
V\left(  \rho\left(  x\right)  \right)  \left\vert f-f\left(  \exp_{O}\left(
Ru\right)  \right)  \right\vert ^{2}\frac{\varphi^{\prime}\left(  \rho\left(
x\right)  \right)  }{\varphi\left(  \rho\left(  x\right)  \right)  }%
\frac{J^{\prime}\left(  u,\rho\right)  }{J\left(  u,\rho\right)  }dV_{g}%
\end{align*}
and%
\begin{align*}
&
{\displaystyle\int\limits_{\mathbb{M}}}
V\left(  \rho\left(  x\right)  \right)  \left\vert \partial_{\rho}\left(
f-f\left(  \exp_{O}\left(  Ru\right)  \right)  \right)  \right\vert _{g}%
^{2}dx-%
{\displaystyle\int\limits_{\mathbb{M}}}
W\left(  \rho\left(  x\right)  \right)  \left\vert f-f\left(  \exp_{O}\left(
Ru\right)  \right)  \right\vert ^{2}dV_{g}\\
&  =%
{\displaystyle\int\limits_{\mathbb{M}}}
V\left(  \rho\left(  x\right)  \right)  \varphi^{2}\left(  \rho\left(
x\right)  \right)  \left\vert \partial_{\rho}\left(  \frac{f-f\left(  \exp
_{O}\left(  Ru\right)  \right)  }{\varphi\left(  \rho\left(  x\right)
\right)  }\right)  \right\vert ^{2}dV_{g}\\
&  -%
{\displaystyle\int\limits_{\mathbb{M}}}
V\left(  \rho\left(  x\right)  \right)  \left\vert f-f\left(  \exp_{O}\left(
Ru\right)  \right)  \right\vert ^{2}\frac{\varphi^{\prime}\left(  \rho\left(
x\right)  \right)  }{\varphi\left(  \rho\left(  x\right)  \right)  }%
\frac{J^{\prime}\left(  u,\rho\right)  }{J\left(  u,\rho\right)  }dV_{g}.
\end{align*}
Here $J^{\prime}\left(  u,t\right)  =\frac{\partial J\left(  u,t\right)
}{\partial t}$ and $x=\exp_{O}\left(  \rho u\right)  $.
\end{theorem}

By applying our main results to some explicit Bessel pairs on Cartan--Hadamard
manifold, we obtain many interesting Hardy identities and inequalities. For
instance, on the hyperbolic space, we obtain the following identities and
inequalities that substantially improve (\ref{Ha1}) as consequences of our
main results:

\begin{theorem}
\label{T3.1}For $f\in C_{0}^{\infty}\left(  \mathbb{H}^{N}\right)  :$%
\begin{align*}
&
{\displaystyle\int\limits_{\mathbb{H}^{N}}}
\left\vert \nabla_{\mathbb{H}}f\right\vert ^{2}dV_{\mathbb{H}}-\left(
\frac{N-2}{2}\right)  ^{2}%
{\displaystyle\int\limits_{\mathbb{H}^{N}}}
\frac{\left\vert f\right\vert ^{2}}{\rho^{2}\left(  x\right)  }dV_{\mathbb{H}%
}\\
&  =%
{\displaystyle\int\limits_{\mathbb{H}^{N}}}
\frac{1}{\rho^{N-2}\left(  x\right)  }\left\vert \nabla_{\mathbb{H}}\left(
\rho^{\frac{N-2}{2}}\left(  x\right)  f\right)  \right\vert ^{2}%
dV_{\mathbb{H}}\\
&  +\frac{\left(  N-2\right)  \left(  N-1\right)  }{2}%
{\displaystyle\int\limits_{\mathbb{H}^{N}}}
\frac{\rho\left(  x\right)  \cosh\rho\left(  x\right)  -\sinh\rho\left(
x\right)  }{\rho^{2}\left(  x\right)  \sinh\rho\left(  x\right)  }\left\vert
f\right\vert ^{2}dV_{\mathbb{H}}%
\end{align*}
and
\begin{align*}
&
{\displaystyle\int\limits_{\mathbb{H}^{N}}}
\left\vert \partial_{\rho}f\right\vert ^{2}dV_{\mathbb{H}}-\left(  \frac
{N-2}{2}\right)  ^{2}%
{\displaystyle\int\limits_{\mathbb{H}^{N}}}
\frac{\left\vert f\right\vert ^{2}}{\rho^{2}\left(  x\right)  }dV_{\mathbb{H}%
}\\
&  =%
{\displaystyle\int\limits_{\mathbb{H}^{N}}}
\frac{1}{\rho^{N-2}\left(  x\right)  }\left\vert \partial_{\rho}\left(
\rho^{\frac{N-2}{2}}\left(  x\right)  f\right)  \right\vert ^{2}%
dV_{\mathbb{H}}\\
&  +\frac{\left(  N-2\right)  \left(  N-1\right)  }{2}%
{\displaystyle\int\limits_{\mathbb{H}^{N}}}
\frac{\rho\left(  x\right)  \cosh\rho\left(  x\right)  -\sinh\rho\left(
x\right)  }{\rho^{2}\left(  x\right)  \sinh\rho\left(  x\right)  }\left\vert
f\right\vert ^{2}dV_{\mathbb{H}}%
\end{align*}

\end{theorem}

Obviously, our theorem gives the exact remainder and therefore provides the
direct understanding for the Hardy inequality (\ref{Ha1}). Also, as a
consequence of the above identities, we get that
\begin{align*}
&
{\displaystyle\int\limits_{\mathbb{H}^{N}}}
\left\vert \nabla_{\mathbb{H}}f\right\vert ^{2}dV_{\mathbb{H}}\geq%
{\displaystyle\int\limits_{\mathbb{H}^{N}}}
\left\vert \partial_{\rho}f\right\vert ^{2}dV_{\mathbb{H}}\\
&  \geq\left(  \frac{N-2}{2}\right)  ^{2}%
{\displaystyle\int\limits_{\mathbb{H}^{N}}}
\frac{\left\vert f\right\vert ^{2}}{\rho^{2}\left(  x\right)  }dV_{\mathbb{H}%
}\\
&  +\frac{\left(  N-2\right)  \left(  N-1\right)  }{2}%
{\displaystyle\int\limits_{\mathbb{H}^{N}}}
\frac{\rho\left(  x\right)  \cosh\rho\left(  x\right)  -\sinh\rho\left(
x\right)  }{\rho^{2}\left(  x\right)  \sinh\rho\left(  x\right)  }\left\vert
f\right\vert ^{2}dV_{\mathbb{H}}.
\end{align*}
Therefore, the operator $-\Delta_{\mathbb{H}}-\left(  \frac{N-2}{2}\right)
^{2}\frac{1}{\rho^{2}\left(  x\right)  }$ is subcritical in $\mathbb{H}%
^{N}\setminus\left\{  0\right\}  $, which is in constrast to the situation in
the Euclidean setting \cite[Corollary 2.3.4]{GM1}.

We also present the exact remainder for the Hardy-Poincar\'{e}--Sobolev
inequality (\ref{Hy2}) and thus sharpen the inequality (\ref{Hy1}) and
illustrate more precise understanding of (\ref{Hy2}):

\begin{theorem}
\label{T3.2}For $f\in C_{0}^{\infty}\left(  \mathbb{H}^{N}\right)  :$%
\begin{align*}
&
{\displaystyle\int\limits_{\mathbb{H}^{N}}}
\left\vert \nabla_{\mathbb{H}}f\right\vert ^{2}dV_{\mathbb{H}}-%
{\displaystyle\int\limits_{\mathbb{H}^{N}}}
\left[  \frac{\left(  N-1\right)  ^{2}}{4}+\frac{1}{4}\frac{1}{\rho^{2}\left(
x\right)  }+\frac{\left(  N-1\right)  \left(  N-3\right)  }{4}\frac{1}%
{\sinh^{2}\rho\left(  x\right)  }\right]  \left\vert f\right\vert
^{2}dV_{\mathbb{H}}\\
&  =%
{\displaystyle\int\limits_{\mathbb{H}^{N}}}
\frac{\rho\left(  x\right)  }{\sinh^{N-1}\rho\left(  x\right)  }\left\vert
\nabla_{\mathbb{H}}\left(  \frac{\sinh^{\frac{N-1}{2}}\rho\left(  x\right)
f}{\rho^{\frac{1}{2}}\left(  x\right)  }\right)  \right\vert ^{2}%
dV_{\mathbb{H}},
\end{align*}
and%
\begin{align*}
&
{\displaystyle\int\limits_{\mathbb{H}^{N}}}
\left\vert \partial_{\rho}f\right\vert ^{2}dV_{\mathbb{H}}-%
{\displaystyle\int\limits_{\mathbb{H}^{N}}}
\left[  \frac{\left(  N-1\right)  ^{2}}{4}+\frac{1}{4}\frac{1}{\rho^{2}\left(
x\right)  }+\frac{\left(  N-1\right)  \left(  N-3\right)  }{4}\frac{1}%
{\sinh^{2}\rho\left(  x\right)  }\right]  \left\vert f\right\vert
^{2}dV_{\mathbb{H}}\\
&  =%
{\displaystyle\int\limits_{\mathbb{H}^{N}}}
\frac{\rho\left(  x\right)  }{\sinh^{N-1}\rho\left(  x\right)  }\left\vert
\partial_{\rho}\left(  \frac{\sinh^{\frac{N-1}{2}}\rho\left(  x\right)
f}{\rho^{\frac{1}{2}}\left(  x\right)  }\right)  \right\vert ^{2}%
dV_{\mathbb{H}}%
\end{align*}
and
\begin{align*}%
{\displaystyle\int\limits_{\mathbb{H}^{N}}}
\left\vert \nabla_{\mathbb{H}}f\right\vert ^{2}dV_{\mathbb{H}}  &  \geq%
{\displaystyle\int\limits_{\mathbb{H}^{N}}}
\left\vert \partial_{\rho}f\right\vert ^{2}dV_{\mathbb{H}}\\
&  \geq%
{\displaystyle\int\limits_{\mathbb{H}^{N}}}
\left[  \frac{\left(  N-1\right)  ^{2}}{4}+\frac{1}{4}\frac{1}{\rho^{2}\left(
x\right)  }+\frac{\left(  N-1\right)  \left(  N-3\right)  }{4}\frac{1}%
{\sinh^{2}\rho\left(  x\right)  }\right]  \left\vert f\right\vert
^{2}dV_{\mathbb{H}}.
\end{align*}

\end{theorem}

We also obtain the following Hardy inequalities on hyperbolic spaces in the
spirit of Brezis-V\'{a}zquez \cite{BV}:

\begin{theorem}
\label{T3.3}Let $0\leq\alpha\leq\frac{N-2}{2}$. For $f\in C_{0}^{\infty
}\left(  \mathbb{H}^{N}\right)  :$
\begin{align*}
&
{\displaystyle\int\limits_{0<\rho\left(  x\right)  <R}}
\left\vert \nabla_{\mathbb{H}}f\right\vert ^{2}dV_{\mathbb{H}}-\left(
\frac{\left(  N-2\right)  ^{2}}{4}-\alpha^{2}\right)
{\displaystyle\int\limits_{0<\rho\left(  x\right)  <R}}
\frac{\left\vert f\right\vert ^{2}}{\rho^{2}\left(  x\right)  }dV_{\mathbb{H}%
}\\
&  =\frac{z_{\alpha}^{2}}{R^{2}}%
{\displaystyle\int\limits_{0<\rho\left(  x\right)  <R}}
\left\vert f\right\vert ^{2}dV_{\mathbb{H}}+%
{\displaystyle\int\limits_{0<\rho\left(  x\right)  <R}}
\frac{J_{\alpha}^{2}\left(  \frac{z_{\alpha}}{R}\rho\left(  x\right)  \right)
}{\rho^{N-2}\left(  x\right)  }\left\vert \nabla_{\mathbb{H}}\left(  \frac
{f}{\rho^{\frac{2-N}{2}}\left(  x\right)  J_{\alpha}\left(  \frac{z_{\alpha}%
}{R}\rho\left(  x\right)  \right)  }\right)  \right\vert ^{2}dV_{\mathbb{H}}\\
&  -\left(  N-1\right)
{\displaystyle\int\limits_{0<\rho\left(  x\right)  <R}}
\left(  \frac{2-N}{2\rho\left(  x\right)  }+\frac{z_{\alpha}}{R}%
\frac{J_{\alpha}^{\prime}\left(  \frac{z_{\alpha}}{R}\rho\left(  x\right)
\right)  }{J_{\alpha}\left(  \frac{z_{\alpha}}{R}\rho\left(  x\right)
\right)  }\right)  \frac{\rho\left(  x\right)  \cosh\rho\left(  x\right)
-\sinh\rho\left(  x\right)  }{\rho\left(  x\right)  \sinh\rho\left(  x\right)
}\left\vert f\right\vert ^{2}dV_{\mathbb{H}}%
\end{align*}
and
\begin{align*}
&
{\displaystyle\int\limits_{0<\rho\left(  x\right)  <R}}
\left\vert \partial_{\rho}f\right\vert ^{2}dV_{\mathbb{H}}-\left(
\frac{\left(  N-2\right)  ^{2}}{4}-\alpha^{2}\right)
{\displaystyle\int\limits_{0<\rho\left(  x\right)  <R}}
\frac{\left\vert f\right\vert ^{2}}{\rho^{2}\left(  x\right)  }dV_{\mathbb{H}%
}\\
&  =\frac{z_{\alpha}^{2}}{R^{2}}%
{\displaystyle\int\limits_{0<\rho\left(  x\right)  <R}}
\left\vert f\right\vert ^{2}dV_{\mathbb{H}}+%
{\displaystyle\int\limits_{0<\rho\left(  x\right)  <R}}
\frac{J_{\alpha}^{2}\left(  \frac{z_{\alpha}}{R}\rho\left(  x\right)  \right)
}{\rho^{N-2}\left(  x\right)  }\left\vert \partial_{\rho}\left(  \frac{f}%
{\rho^{\frac{2-N}{2}}\left(  x\right)  J_{\alpha}\left(  \frac{z_{\alpha}}%
{R}\rho\left(  x\right)  \right)  }\right)  \right\vert ^{2}dV_{\mathbb{H}}\\
&  -\left(  N-1\right)
{\displaystyle\int\limits_{0<\rho\left(  x\right)  <R}}
\left(  \frac{2-N}{2\rho\left(  x\right)  }+\frac{z_{\alpha}}{R}%
\frac{J_{\alpha}^{\prime}\left(  \frac{z_{\alpha}}{R}\rho\left(  x\right)
\right)  }{J_{\alpha}\left(  \frac{z_{\alpha}}{R}\rho\left(  x\right)
\right)  }\right)  \frac{\rho\left(  x\right)  \cosh\rho\left(  x\right)
-\sinh\rho\left(  x\right)  }{\rho\left(  x\right)  \sinh\rho\left(  x\right)
}\left\vert f\right\vert ^{2}dV_{\mathbb{H}}%
\end{align*}
As a consequence of these identities, we get that%
\begin{align*}
&
{\displaystyle\int\limits_{0<\rho\left(  x\right)  <R}}
\left\vert \nabla_{\mathbb{H}}f\right\vert ^{2}dV_{\mathbb{H}}\\
&  \geq%
{\displaystyle\int\limits_{0<\rho\left(  x\right)  <R}}
\left\vert \partial_{\rho}f\right\vert ^{2}dV_{\mathbb{H}}\\
&  \geq\left(  \frac{\left(  N-2\right)  ^{2}}{4}-\alpha^{2}\right)
{\displaystyle\int\limits_{0<\rho\left(  x\right)  <R}}
\frac{\left\vert f\right\vert ^{2}}{\rho^{2}\left(  x\right)  }dV_{\mathbb{H}%
}+\frac{z_{\alpha}^{2}}{R^{2}}%
{\displaystyle\int\limits_{0<\rho\left(  x\right)  <R}}
\left\vert f\right\vert ^{2}dV_{\mathbb{H}}\\
&  -\left(  N-1\right)
{\displaystyle\int\limits_{0<\rho\left(  x\right)  <R}}
\left(  \frac{2-N}{2\rho\left(  x\right)  }+\frac{z_{\alpha}}{R}%
\frac{J_{\alpha}^{\prime}\left(  \frac{z_{\alpha}}{R}\rho\left(  x\right)
\right)  }{J_{\alpha}\left(  \frac{z_{\alpha}}{R}\rho\left(  x\right)
\right)  }\right)  \frac{\rho\left(  x\right)  \cosh\rho\left(  x\right)
-\sinh\rho\left(  x\right)  }{\rho\left(  x\right)  \sinh\rho\left(  x\right)
}\left\vert f\right\vert ^{2}dV_{\mathbb{H}}\\
&  \geq\left(  \frac{\left(  N-2\right)  ^{2}}{4}-\alpha^{2}\right)
{\displaystyle\int\limits_{0<\rho\left(  x\right)  <R}}
\frac{\left\vert f\right\vert ^{2}}{\rho^{2}\left(  x\right)  }dV_{\mathbb{H}%
}+\frac{z_{\alpha}^{2}}{R^{2}}%
{\displaystyle\int\limits_{0<\rho\left(  x\right)  <R}}
\left\vert f\right\vert ^{2}dV_{\mathbb{H}}.
\end{align*}
Here $z_{\alpha}$ is the first zero of the Bessel function of the first kind
$J_{\alpha}\left(  z\right)  $.
\end{theorem}

\vskip0.5cm

Our paper is organized as follows: In Section 2, we use our main results on
the Hardy identities to obtain several Hardy type inequalities and their
improvements on Cartan-Hadamard manifolds. In Section 3, we will focus on
deriving the Hardy identities and inequalities on hyperbolic spaces. We also
provide the proofs of Theorem \ref{T3.1}, Theorem \ref{T3.2} and Theorem
\ref{T3.3} in this section. Proofs of main results (Theorem \ref{T1} and
Theorem \ref{T2}) will be presented in Section 4.

\section{Hardy inequalities on Cartan-Hadamard manifolds}

We note that if the sectional curvature $K_{\mathbb{M}}=-b$, then $J\left(
u,t\right)  =J_{b}\left(  t\right)  $ does not depend on $u$. Moreover%
\[
J_{b}\left(  t\right)  =\left\{
\begin{array}
[c]{ll}%
1 & \text{if }b=0\\
\left(  \frac{\sinh\left(  \sqrt{b}t\right)  }{\sqrt{b}t}\right)  ^{N-1} &
\text{if }b>0
\end{array}
\right.  .
\]
Also, if $K_{\mathbb{M}}\leq-b\leq0$, then by the Bishop-Gromov-G\"{u}nther
comparison theorem \cite[page 172]{GHL}, we have that%
\[
\frac{J^{\prime}\left(  u,t\right)  }{J\left(  u,t\right)  }\geq\frac
{J_{b}^{\prime}\left(  t\right)  }{J_{b}\left(  t\right)  }=\frac{N-1}%
{t}\mathbf{D}_{b}\left(  t\right)
\]
where $J^{\prime}\left(  u,t\right)  =\frac{\partial J\left(  u,t\right)
}{\partial t}$,%
\[
\mathbf{D}_{b}\left(  t\right)  =\left\{
\begin{array}
[c]{ll}%
0 & \text{if }t=0\\
t\mathbf{ct}_{b}\left(  t\right)  -1 & \text{if }t>0
\end{array}
\right.
\]
and
\[
\mathbf{ct}_{b}\left(  t\right)  =\left\{
\begin{array}
[c]{ll}%
\frac{1}{t} & \text{if }b=0\\
\sqrt{b}\coth\left(  \sqrt{b}t\right)  & \text{if }b>0
\end{array}
\right.  .
\]
Therefore, we obtain the following Hardy type inequality as a direct
consequence of our Theorem \ref{T1}:

\begin{theorem}
\label{CT1} Let $(\mathbb{M},g)$ be a Cartan-Hadamard manifold of dimension
$N$ and let $O\in\mathbb{M}$. \textit{Let }$0<R\leq\infty$\textit{, }%
$V$\textit{ and }$W$\textit{ be positive }$C^{1}-$\textit{functions on
}$\left(  0,R\right)  $ such that $\left(  r^{N-1}V,r^{N-1}W\right)  $ is a
Bessel pair on $\left(  0,R\right)  $ with nonincreasing positive solution
$\varphi$. Then for $f\in C_{0}^{\infty}\left(  B_{R}\left(  O\right)
\setminus\rho^{-1}\left\{  0\right\}  \right)  :$
\begin{align*}
&
{\displaystyle\int\limits_{B_{R}\left(  O\right)  }}
V\left(  \rho\left(  x\right)  \right)  \left\vert \nabla_{g}f\right\vert
_{g}^{2}dV_{g}-%
{\displaystyle\int\limits_{B_{R}\left(  O\right)  }}
W\left(  \rho\left(  x\right)  \right)  \left\vert f\right\vert ^{2}dV_{g}\\
&  \geq%
{\displaystyle\int\limits_{B_{R}\left(  O\right)  }}
V\left(  \rho\left(  x\right)  \right)  \left\vert \varphi^{2}\left(
\rho\left(  x\right)  \right)  \right\vert \left\vert \nabla_{g}\left(
\frac{f}{\varphi\left(  \rho\left(  x\right)  \right)  }\right)  \right\vert
_{g}^{2}dV_{g}\\
&  -%
{\displaystyle\int\limits_{B_{R}\left(  O\right)  }}
V\left(  \rho\left(  x\right)  \right)  \left\vert f\right\vert ^{2}%
\frac{\varphi^{\prime}\left(  \rho\left(  x\right)  \right)  }{\varphi\left(
\rho\left(  x\right)  \right)  }\frac{J_{b}^{\prime}\left(  \rho\left(
x\right)  \right)  }{J_{b}\left(  \rho\left(  x\right)  \right)  }dV_{g}\\
&  \geq%
{\displaystyle\int\limits_{B_{R}\left(  O\right)  }}
V\left(  \rho\left(  x\right)  \right)  \left\vert \varphi^{2}\left(
\rho\left(  x\right)  \right)  \right\vert \left\vert \nabla_{g}\left(
\frac{f}{\varphi\left(  \rho\left(  x\right)  \right)  }\right)  \right\vert
_{g}^{2}dV_{g}%
\end{align*}
and
\begin{align*}
&
{\displaystyle\int\limits_{B_{R}\left(  O\right)  }}
V\left(  \rho\left(  x\right)  \right)  \left\vert \partial_{\rho}f\right\vert
^{2}dV_{g}-%
{\displaystyle\int\limits_{B_{R}\left(  O\right)  }}
W\left(  \rho\left(  x\right)  \right)  \left\vert f\right\vert ^{2}dV_{g}\\
&  \geq%
{\displaystyle\int\limits_{B_{R}\left(  O\right)  }}
V\left(  \rho\left(  x\right)  \right)  \varphi^{2}\left(  \rho\left(
x\right)  \right)  \left\vert \partial_{\rho}\left(  \frac{f}{\varphi\left(
\rho\left(  x\right)  \right)  }\right)  \right\vert ^{2}dV_{g}\\
&  -%
{\displaystyle\int\limits_{B_{R}\left(  O\right)  }}
V\left(  \rho\left(  x\right)  \right)  \left\vert f\right\vert ^{2}%
\frac{\varphi^{\prime}\left(  \rho\left(  x\right)  \right)  }{\varphi\left(
\rho\left(  x\right)  \right)  }\frac{J_{b}^{\prime}\left(  \rho\left(
x\right)  \right)  }{J_{b}\left(  \rho\left(  x\right)  \right)  }dV_{g}\\
&  \geq%
{\displaystyle\int\limits_{B_{R}\left(  O\right)  }}
V\left(  \rho\left(  x\right)  \right)  \varphi^{2}\left(  \rho\left(
x\right)  \right)  \left\vert \partial_{\rho}\left(  \frac{f}{\varphi\left(
\rho\left(  x\right)  \right)  }\right)  \right\vert ^{2}dV_{g}.
\end{align*}

\end{theorem}

\begin{proof}
By the Bishop-Gromov-G\"{u}nther comparison theorem \cite[page 172]{GHL}, we
get\ $\frac{J^{\prime}\left(  u,t\right)  }{J\left(  u,t\right)  }\geq0$,
where $J^{\prime}\left(  u,t\right)  =\frac{\partial J\left(  u,t\right)
}{\partial t}.$ Since $\varphi$ is a nonincreasing function, $-\varphi
^{\prime}\left(  \rho\left(  x\right)  \right)  \frac{J^{\prime}\left(
u,\rho\right)  }{J\left(  u,\rho\right)  }\geq-\varphi^{\prime}\left(
\rho\left(  x\right)  \right)  \frac{J_{b}^{\prime}\left(  \rho\right)
}{J_{b}\left(  \rho\right)  }\geq0$. Hence, we can apply Theorem \ref{T1} to
get the desired result.
\end{proof}

By applying Theorem \ref{CT1} to particular Bessel pairs, we obtain several
Hardy type inequalities with remainder terms on $\mathbb{M}$. These results
are listed as follows.

\begin{corollary}
\label{C1}for $\lambda<N-2$ and $f\in C_{0}^{\infty}\left(  \mathbb{M}%
\setminus\rho^{-1}\left\{  0\right\}  \right)  :$%
\begin{align*}
&
{\displaystyle\int\limits_{\mathbb{M}}}
\frac{\left\vert \nabla_{g}f\right\vert _{g}^{2}}{\rho^{\lambda}\left(
x\right)  }dV_{g}-\left(  \frac{N-\lambda-2}{2}\right)  ^{2}%
{\displaystyle\int\limits_{\mathbb{M}}}
\frac{\left\vert f\right\vert ^{2}}{\rho^{\lambda+2}\left(  x\right)  }%
dV_{g}\\
&  \geq%
{\displaystyle\int\limits_{\mathbb{M}}}
\frac{1}{\rho^{N-2}\left(  x\right)  }\left\vert \nabla_{g}\left(  \rho
^{\frac{N-\lambda-2}{2}}\left(  x\right)  f\right)  \right\vert _{g}^{2}dV_{g}%
\end{align*}
and
\begin{align}
&
{\displaystyle\int\limits_{\mathbb{M}}}
\frac{\left\vert \partial_{\rho}f\right\vert ^{2}}{\rho^{\lambda}\left(
x\right)  }dV_{g}-\left(  \frac{N-\lambda-2}{2}\right)  ^{2}%
{\displaystyle\int\limits_{\mathbb{M}}}
\frac{\left\vert f\right\vert ^{2}}{\rho^{\lambda+2}\left(  x\right)  }%
dV_{g}\label{C1.1}\\
&  \geq%
{\displaystyle\int\limits_{\mathbb{M}}}
\frac{1}{\rho^{N-2}\left(  x\right)  }\left\vert \partial_{\rho}\left(
\rho^{\frac{N-\lambda-2}{2}}\left(  x\right)  f\right)  \right\vert ^{2}%
dV_{g}.\nonumber
\end{align}

\end{corollary}

\begin{proof}
We apply Theorem \ref{CT1} to the Bessel pair $\left(  r^{N-1}r^{-\lambda
},r^{N-1}r^{-\lambda}\frac{\left(  N-\lambda-2\right)  ^{2}}{4}\frac{1}{r^{2}%
}\right)  $ on $\left(  0,\infty\right)  $ with $\varphi\left(  r\right)
=r^{\frac{2-N+\lambda}{2}}$ to get the desired results.
\end{proof}

In the critical case $\lambda=N-2$, we have

\begin{corollary}
\label{C2}Let $R>0$. We have for $f\in C_{0}^{\infty}\left(  B_{R}\left(
O\right)  \setminus\rho^{-1}\left\{  0\right\}  \right)  :$
\begin{align*}
&
{\displaystyle\int\limits_{0<\rho\left(  x\right)  <R}}
\frac{\left\vert \nabla_{g}f\right\vert _{g}^{2}}{\rho^{N-2}\left(  x\right)
}dV_{g}-\frac{1}{4}%
{\displaystyle\int\limits_{0<\rho\left(  x\right)  <R}}
\frac{\left\vert f\left(  x\right)  \right\vert ^{2}}{\rho^{N}\left(
x\right)  \left\vert \ln\frac{\rho\left(  x\right)  }{R}\right\vert ^{2}%
}dV_{g}\\
&  \geq%
{\displaystyle\int\limits_{0<\rho\left(  x\right)  <R}}
\frac{1}{\rho^{N-2}\left(  x\right)  }\left\vert \ln\frac{\rho\left(
x\right)  }{R}\right\vert \left\vert \nabla_{g}\left(  \frac{f\left(
x\right)  }{\sqrt{\left\vert \ln\frac{\left\vert x\right\vert }{R}\right\vert
}}\right)  \right\vert _{g}^{2}dV_{g}%
\end{align*}
and%
\begin{align*}
&
{\displaystyle\int\limits_{0<\rho\left(  x\right)  <R}}
\frac{\left\vert \partial_{\rho}f\right\vert ^{2}}{\rho^{N-2}\left(  x\right)
}dV_{g}-\frac{1}{4}%
{\displaystyle\int\limits_{0<\rho\left(  x\right)  <R}}
\frac{\left\vert f\left(  x\right)  \right\vert ^{2}}{\rho^{N}\left(
x\right)  \left\vert \ln\frac{\rho\left(  x\right)  }{R}\right\vert ^{2}%
}dV_{g}\\
&  \geq%
{\displaystyle\int\limits_{0<\rho\left(  x\right)  <R}}
\frac{1}{\rho^{N-2}\left(  x\right)  }\left\vert \ln\frac{\rho\left(
x\right)  }{R}\right\vert \left\vert \partial_{\rho}\left(  \frac{f\left(
x\right)  }{\sqrt{\left\vert \ln\frac{\left\vert x\right\vert }{R}\right\vert
}}\right)  \right\vert ^{2}dV_{g}.
\end{align*}

\end{corollary}

\begin{proof}
We apply Theorem \ref{CT1} to the Bessel pair $\left(  r^{N-1}\frac{1}%
{r^{N-2}},r^{N-1}\frac{1}{4r^{N}\left\vert \ln\frac{r}{R}\right\vert ^{2}%
}\right)  $ with $\varphi=\sqrt{\left\vert \ln\frac{r}{R}\right\vert }$.
\end{proof}

Actually, we can get the following version of the critical Hardy type
inequalities on the whole space $\mathbb{M}$ which is more general than
Corollary \ref{C2}:

\begin{corollary}
\label{C3}Let $R>0$. For any $f\in C_{0}^{\infty}\left(  \mathbb{M}%
\setminus\rho^{-1}\left\{  0\right\}  \right)  $, we have
\begin{align*}
&
{\displaystyle\int\limits_{\mathbb{M}}}
\frac{1}{\rho^{N-2}\left(  x\right)  }\left\vert \nabla_{g}\left(  f\left(
x\right)  -f\left(  \exp_{O}\left(  Ru\right)  \right)  \right)  \right\vert
_{g}^{2}dV_{g}\\
&  \geq\frac{1}{4}%
{\displaystyle\int\limits_{\mathbb{M}}}
\frac{\left\vert f\left(  x\right)  -f\left(  \exp_{O}\left(  Ru\right)
\right)  \right\vert ^{2}}{\rho^{N}\left(  x\right)  \left\vert \ln\frac
{\rho\left(  x\right)  }{R}\right\vert ^{2}}dV_{g}\\
&  +%
{\displaystyle\int\limits_{\mathbb{M}}}
\frac{1}{\rho^{N-2}\left(  x\right)  }\left\vert \ln\frac{\rho\left(
x\right)  }{R}\right\vert \left\vert \nabla\left(  \frac{f\left(  x\right)
-f\left(  \exp_{O}\left(  Ru\right)  \right)  }{\sqrt{\left\vert \ln
\frac{\left\vert x\right\vert }{R}\right\vert }}\right)  \right\vert _{g}%
^{2}dV_{g}%
\end{align*}
and%
\begin{align*}
&
{\displaystyle\int\limits_{\mathbb{M}}}
\frac{1}{\rho^{N-2}\left(  x\right)  }\left\vert \partial_{\rho}\left(
f\left(  x\right)  -f\left(  \exp_{O}\left(  Ru\right)  \right)  \right)
\right\vert ^{2}dV_{g}\\
&  \geq\frac{1}{4}%
{\displaystyle\int\limits_{\mathbb{M}}}
\frac{\left\vert f\left(  x\right)  -f\left(  \exp_{O}\left(  Ru\right)
\right)  \right\vert ^{2}}{\rho^{N}\left(  x\right)  \left\vert \ln\frac
{\rho\left(  x\right)  }{R}\right\vert ^{2}}dV_{g}\\
&  +%
{\displaystyle\int\limits_{\mathbb{M}}}
\frac{1}{\rho^{N-2}\left(  x\right)  }\left\vert \ln\frac{\rho\left(
x\right)  }{R}\right\vert \left\vert \partial_{\rho}\left(  \frac{f\left(
x\right)  -f\left(  \exp_{O}\left(  Ru\right)  \right)  }{\sqrt{\left\vert
\ln\frac{\left\vert x\right\vert }{R}\right\vert }}\right)  \right\vert
^{2}dV_{g}.
\end{align*}

\end{corollary}

\begin{proof}
Apply the Theorem \ref{T2} to $V\left(  r\right)  =\frac{1}{r^{N-2}}$,
$W=\frac{1}{4r^{N}\left\vert \ln\frac{r}{R}\right\vert ^{2}}$ and
$\varphi=\sqrt{\left\vert \ln\frac{r}{R}\right\vert }$. Note that for any
$f\in C_{0}^{\infty}\left(  \mathbb{M}\setminus\rho^{-1}\left\{  0\right\}
\right)  $ and any $u\in\mathbb{S}^{N-1}:$%
\begin{align*}
&  \lim_{r\rightarrow R}\left\vert V\left(  r\right)  \frac{\varphi^{\prime
}\left(  r\right)  }{\varphi\left(  r\right)  }\left\vert f\left(  \exp
_{O}\left(  ru\right)  \right)  -f\left(  \exp_{O}\left(  Ru\right)  \right)
\right\vert ^{2}\right\vert \\
&  =\lim_{r\rightarrow R}\frac{1}{r^{N-2}}\left\vert \frac{\left(
\sqrt{\left\vert \ln\frac{r}{R}\right\vert }\right)  ^{\prime}}{\sqrt
{\left\vert \ln\frac{r}{R}\right\vert }}\right\vert \left\vert f\left(
\exp_{O}\left(  ru\right)  \right)  -f\left(  \exp_{O}\left(  Ru\right)
\right)  \right\vert ^{2}\\
&  \lesssim\lim_{r\rightarrow R}\frac{1}{\left\vert \ln\frac{r}{R}\right\vert
}\left(  R-r\right)  ^{2}=0.
\end{align*}

\end{proof}

Now, by combining Corollaries \ref{C1} and \ref{C3}, we obtain

\begin{corollary}
\label{C4} Let $(\mathbb{M},g)$ be a Cartan-Hadamard manifold of dimension
$N$. For $\lambda<N-2$ and $f\in C_{0}^{\infty}\left(  \mathbb{M}\setminus
\rho^{-1}\left\{  0\right\}  \right)  $, we have%
\begin{align}
&
{\displaystyle\int\limits_{\mathbb{M}}}
\frac{\left\vert \partial_{\rho}f\right\vert ^{2}}{\rho^{\lambda}\left(
x\right)  }dV_{g}-\left(  \frac{N-\lambda-2}{2}\right)  ^{2}%
{\displaystyle\int\limits_{\mathbb{M}}}
\frac{\left\vert f\right\vert ^{2}}{\rho^{\lambda+2}\left(  x\right)  }%
dV_{g}\nonumber\\
&  \geq\frac{1}{4}\sup_{R>0}%
{\displaystyle\int\limits_{\mathbb{M}}}
\frac{\left\vert f\left(  x\right)  -R^{\frac{N-\lambda-2}{2}}f\left(
\exp_{O}\left(  Ru\right)  \right)  \rho^{-\frac{N-\lambda-2}{2}}\left(
x\right)  \right\vert ^{2}}{\rho^{\lambda+2}\left(  x\right)  \left\vert
\ln\frac{\rho\left(  x\right)  }{R}\right\vert ^{2}}dV_{g}. \label{C4.1}%
\end{align}

\end{corollary}

\begin{proof}
From Corollaries \ref{C1} and \ref{C3}, we have
\begin{align*}
&
{\displaystyle\int\limits_{\mathbb{M}}}
\frac{\left\vert \partial_{\rho}f\right\vert ^{2}}{\rho^{\lambda}\left(
x\right)  }dV_{g}-\left(  \frac{N-\lambda-2}{2}\right)  ^{2}%
{\displaystyle\int\limits_{\mathbb{M}}}
\frac{\left\vert f\right\vert ^{2}}{\rho^{\lambda+2}\left(  x\right)  }%
dV_{g}\\
&  \geq%
{\displaystyle\int\limits_{\mathbb{M}}}
\frac{1}{\rho^{N-2}\left(  x\right)  }\left\vert \partial_{\rho}\left(
\rho^{\frac{N-\lambda-2}{2}}\left(  x\right)  f\right)  \right\vert ^{2}%
dV_{g}\\
&  =%
{\displaystyle\int\limits_{\mathbb{M}}}
\frac{1}{\rho^{N-2}\left(  x\right)  }\left\vert \partial_{\rho}\left(
\rho^{\frac{N-\lambda-2}{2}}\left(  x\right)  f-R^{\frac{N-\lambda-2}{2}%
}f\left(  \exp_{O}\left(  Ru\right)  \right)  \right)  \right\vert ^{2}%
dV_{g}\\
&  \geq\frac{1}{4}%
{\displaystyle\int\limits_{\mathbb{M}}}
\frac{\left\vert \rho^{\frac{N-\lambda-2}{2}}\left(  x\right)  f\left(
x\right)  -R^{\frac{N-\lambda-2}{2}}f\left(  \exp_{O}\left(  Ru\right)
\right)  \right\vert ^{2}}{\rho^{N}\left(  x\right)  \left\vert \ln\frac
{\rho\left(  x\right)  }{R}\right\vert ^{2}}dV_{g}\\
&  +%
{\displaystyle\int\limits_{\mathbb{M}}}
\frac{1}{\rho^{N-2}\left(  x\right)  }\left\vert \ln\frac{\rho\left(
x\right)  }{R}\right\vert \left\vert \partial_{\rho}\left(  \frac{\rho
^{\frac{N-\lambda-2}{2}}f\left(  x\right)  -R^{\frac{N-\lambda-2}{2}}f\left(
\exp_{O}\left(  Ru\right)  \right)  }{\sqrt{\left\vert \ln\frac{\left\vert
x\right\vert }{R}\right\vert }}\right)  \right\vert ^{2}dV_{g}.
\end{align*}

\end{proof}

Note that the \textquotedblleft virtual\textquotedblright\ optimizers of the
weighted Hardy inequality (\ref{C1.1}) have the form $\psi\left(  \exp
_{O}\left(  u\right)  \right)  \rho^{-\frac{N-\lambda-2}{2}}\left(  x\right)
$ for some function $\psi:\mathbb{S}^{N-1}\rightarrow%
\mathbb{R}
$. These optimizers are virtual in the sense that, if equality were to hold in
the Hardy inequality
\begin{align*}
&
{\displaystyle\int\limits_{\mathbb{M}}}
\frac{\left\vert \partial_{\rho}f\right\vert ^{2}}{\rho^{\lambda}\left(
x\right)  }dV_{g}\geq\left(  \frac{N-\lambda-2}{2}\right)  ^{2}%
{\displaystyle\int\limits_{\mathbb{M}}}
\frac{\left\vert f\right\vert ^{2}}{\rho^{\lambda+2}\left(  x\right)  }dV_{g}%
\end{align*}
given by \eqref{C1.1}, then the remainder term would vanish, i.e.,
$\partial_{\rho}\left(  \rho^{\frac{N-\lambda-2}{2}}(x)f \right)  =0$.
Therefore, (\ref{C4.1}) can be read as a stability version of the weighted
Hardy inequality (\ref{C1.1}).

\vskip0.5cm

We also obtain the following Hardy inequality in the spirit of Brezis and
V\'{a}zquez \cite{BV}:

\begin{corollary}
Let $(\mathbb{M},g)$ be a Cartan-Hadamard manifold of dimension $N$. For any
$R>0$ and $\lambda\leq N-2$, we have for $f\in C_{0}^{\infty}\left(
B_{R}\left(  O\right)  \setminus\rho^{-1}\left\{  0\right\}  \right)  :$%
\begin{align*}
&
{\displaystyle\int\limits_{B_{R}\left(  O\right)  }}
\frac{\left\vert \nabla_{g}f\right\vert _{g}^{2}}{\rho^{\lambda}\left(
x\right)  }dV_{g}-\left(  \frac{N-\lambda-2}{2}\right)  ^{2}%
{\displaystyle\int\limits_{B_{R}\left(  O\right)  }}
\frac{\left\vert f\right\vert ^{2}}{\rho^{\lambda+2}\left(  x\right)  }%
dV_{g}\\
&  \geq\frac{z_{0}^{2}}{R^{2}}%
{\displaystyle\int\limits_{B_{R}\left(  O\right)  }}
\frac{\left\vert f\right\vert ^{2}}{\rho^{\lambda}\left(  x\right)  }dV_{g}+%
{\displaystyle\int\limits_{B_{R}\left(  O\right)  }}
\left\vert \frac{J_{0}\left(  \frac{z_{0}}{R}\rho\left(  x\right)  \right)
}{\rho\left(  x\right)  ^{\frac{N-\lambda-2}{2}}}\right\vert ^{2}\left\vert
\nabla_{g}\left(  \frac{\rho\left(  x\right)  ^{\frac{N-\lambda-2}{2}}}%
{J_{0}(\frac{z_{0}}{R}\rho\left(  x\right)  )}f\right)  \right\vert _{g}%
^{2}dV_{g}%
\end{align*}
and%
\begin{align*}
&
{\displaystyle\int\limits_{B_{R}\left(  O\right)  }}
\frac{\left\vert \partial_{\rho}f\right\vert ^{2}}{\rho^{\lambda}\left(
x\right)  }dV_{g}-\left(  \frac{N-\lambda-2}{2}\right)  ^{2}%
{\displaystyle\int\limits_{B_{R}\left(  O\right)  }}
\frac{\left\vert f\right\vert ^{2}}{\rho^{\lambda+2}\left(  x\right)  }%
dV_{g}\\
&  \geq\frac{z_{0}^{2}}{R^{2}}%
{\displaystyle\int\limits_{B_{R}\left(  O\right)  }}
\frac{\left\vert f\right\vert ^{2}}{\rho^{\lambda}\left(  x\right)  }dV_{g}+%
{\displaystyle\int\limits_{B_{R}\left(  O\right)  }}
\left\vert \frac{J_{0}\left(  \frac{z_{0}}{R}\rho\left(  x\right)  \right)
}{\rho\left(  x\right)  ^{\frac{N-\lambda-2}{2}}}\right\vert ^{2}\left\vert
\partial_{\rho}\left(  \frac{\rho\left(  x\right)  ^{\frac{N-\lambda-2}{2}}%
}{J_{0}(\frac{z_{0}}{R}\rho\left(  x\right)  )}f\right)  \right\vert
^{2}dV_{g}.
\end{align*}

\end{corollary}

\begin{proof}
For any $R>0$, $\left(  r^{N-1}r^{-\lambda}\text{, }r^{N-1}r^{-\lambda}\left[
\frac{\left(  N-\lambda-2\right)  ^{2}}{4}\frac{1}{r^{2}}+\frac{z_{0}^{2}%
}{R^{2}}\right]  \right)  $ is a Bessel pair on $\left(  0,R\right)  $ with
$\varphi\left(  r\right)  =r^{-\frac{N-\lambda-2}{2}}J_{0}\left(  \frac
{rz_{0}}{R}\right)  =r^{-\frac{N-\lambda-2}{2}}J_{0;R}\left(  r\right)  $.
Here $z_{0}=2.4048...$ is the first zero of the Bessel function $J_{0}\left(
z\right)  $. Note that $\varphi\left(  r\right)  $ is nonincreasing since
$N-\lambda-2\geq0$.
\end{proof}

\section{Hardy inequalities on hyperbolic spaces}

In this section, we will investigate the Hardy identities and inequalities on
the hyperbolic space $\mathbb{H}^{N}$, which is the most important example of
Cartan-Hadamard manifold. We use the Poincar\'{e} ball model of the hyperbolic
space $\mathbb{H}^{N}$. That is, the unit ball in $%
\mathbb{R}
^{N}$ centered at the origin and equipped with the metric
\[
ds^{2}=\frac{4%
{\displaystyle\sum}
dx_{i}^{2}}{\left(  1-r^{2}\right)  ^{2}}.
\]
Also
\begin{align*}
dV_{\mathbb{H}}  &  =\frac{2^{N}}{\left(  1-r^{2}\right)  ^{N}}dx,\\
\nabla_{\mathbb{H}}  &  =\left(  \frac{1-r^{2}}{2}\right)  ^{2}\nabla,
\end{align*}
where $\nabla$ denotes the Euclidean gradient. Therefore%
\[%
{\displaystyle\int\limits_{\mathbb{H}^{N}}}
\left\vert \nabla_{\mathbb{H}}u\right\vert ^{2}dV_{\mathbb{H}}=%
{\displaystyle\int\limits_{B}}
\left\vert \nabla u\right\vert ^{2}\frac{2^{N-2}}{\left(  1-\left\vert
x\right\vert ^{2}\right)  ^{N-2}}dx.
\]
We also recall that the geodesic distance from $x$ to $0$ is $\rho\left(
x\right)  =\ln\frac{1+\left\vert x\right\vert }{1-\left\vert x\right\vert }$.
That is $\left\vert x\right\vert =\frac{e^{\rho\left(  x\right)  }-1}%
{e^{\rho\left(  x\right)  }+1}$. By using the Poincar\'{e} ball model and
applying our Theorem \ref{T1} for the unit ball on the Euclidean space $%
\mathbb{R}
^{N}$, we get the following identity:

\begin{theorem}
\label{T6}If $\left(  r^{N-1}\frac{1}{\left(  1-r^{2}\right)  ^{N-2}}%
V,r^{N-1}\frac{4}{\left(  1-r^{2}\right)  ^{N}}W\right)  $ is a Bessel pair on
$\left(  0,1\right)  $, then we have for $f\in C_{0}^{\infty}\left(
\mathbb{H}^{N}\setminus\left\{  0\right\}  \right)  :$%
\begin{align*}
&
{\displaystyle\int\limits_{\mathbb{H}^{N}}}
V\left(  \frac{e^{\rho\left(  x\right)  }-1}{e^{\rho\left(  x\right)  }%
+1}\right)  \left\vert \nabla_{\mathbb{H}}f\right\vert ^{2}dV_{\mathbb{H}}-%
{\displaystyle\int\limits_{\mathbb{H}^{N}}}
W\left(  \frac{e^{\rho\left(  x\right)  }-1}{e^{\rho\left(  x\right)  }%
+1}\right)  \left\vert f\right\vert ^{2}dV_{\mathbb{H}}\\
&  =%
{\displaystyle\int\limits_{\mathbb{H}^{N}}}
V\left(  \frac{e^{\rho\left(  x\right)  }-1}{e^{\rho\left(  x\right)  }%
+1}\right)  \left\vert \varphi^{2}\left(  \frac{e^{\rho\left(  x\right)  }%
-1}{e^{\rho\left(  x\right)  }+1}\right)  \right\vert \left\vert
\nabla_{\mathbb{H}}\left(  \frac{f}{\varphi\left(  \frac{e^{\rho\left(
x\right)  }-1}{e^{\rho\left(  x\right)  }+1}\right)  }\right)  \right\vert
^{2}dV_{\mathbb{H}}%
\end{align*}
Here $\varphi$ is the positive solution of%
\[
\left(  r^{N-1}\frac{1}{\left(  1-r^{2}\right)  ^{N-2}}V\left(  r\right)
\varphi^{\prime}(r)\right)  ^{\prime}+r^{N-1}\frac{4}{\left(  1-r^{2}\right)
^{N}}W\left(  r\right)  \varphi(r)=0
\]
on $\left(  0,1\right)  $.
\end{theorem}

\begin{proof}
Since $\left(  r^{N-1}\frac{1}{\left(  1-r^{2}\right)  ^{N-2}}V,r^{N-1}%
\frac{4}{\left(  1-r^{2}\right)  ^{N}}W\right)  $ is a Bessel pair on $\left(
0,1\right)  $ with solution $\varphi$, we get%
\begin{align*}
&
{\displaystyle\int\limits_{B\left(  0,1\right)  }}
\frac{1}{\left(  1-\left\vert x\right\vert ^{2}\right)  ^{N-2}}V\left(
\left\vert x\right\vert \right)  \left\vert \nabla f\right\vert ^{2}dx-%
{\displaystyle\int\limits_{B\left(  0,1\right)  }}
\frac{4}{\left(  1-\left\vert x\right\vert ^{2}\right)  ^{N}}W\left(
\left\vert x\right\vert \right)  \left\vert f\right\vert ^{2}dx\\
&  =%
{\displaystyle\int\limits_{B\left(  0,1\right)  }}
\frac{1}{\left(  1-\left\vert x\right\vert ^{2}\right)  ^{N-2}}V\left(
\left\vert x\right\vert \right)  \left\vert \varphi^{2}\left(  \left\vert
x\right\vert \right)  \right\vert \left\vert \nabla\left(  \frac{f}%
{\varphi\left(  \left\vert x\right\vert \right)  }\right)  \right\vert ^{2}dx.
\end{align*}
Equivalently,%
\begin{align*}
&
{\displaystyle\int\limits_{\mathbb{H}^{N}}}
V\left(  \left\vert x\right\vert \right)  \left\vert \nabla_{\mathbb{H}%
}f\right\vert ^{2}dV_{\mathbb{H}}-%
{\displaystyle\int\limits_{\mathbb{H}^{N}}}
W\left(  \left\vert x\right\vert \right)  \left\vert f\right\vert
^{2}dV_{\mathbb{H}}\\
&  =%
{\displaystyle\int\limits_{\mathbb{H}^{N}}}
V\left(  \left\vert x\right\vert \right)  \left\vert \varphi^{2}\left(
\left\vert x\right\vert \right)  \right\vert \left\vert \nabla_{\mathbb{H}%
}\left(  \frac{f}{\varphi\left(  \left\vert x\right\vert \right)  }\right)
\right\vert ^{2}dV_{\mathbb{H}}.
\end{align*}

\end{proof}

Now, let
\begin{align*}
F\left(  r\right)   &  =\frac{\left(  1-r^{2}\right)  ^{N-2}}{r^{N-1}}\\
G\left(  r\right)   &  =%
{\displaystyle\int\limits_{r}^{1}}
F\left(  t\right)  dt\\
V_{2}\left(  r\right)   &  =\frac{F^{2}\left(  r\right)  \left(
1-r^{2}\right)  ^{2}}{4\left(  N-2\right)  ^{2}G^{2}\left(  r\right)  }.
\end{align*}
Then we have

\begin{corollary}
For $f\in C_{0}^{\infty}\left(  \mathbb{H}^{N}\right)  :$
\begin{align*}
&
{\displaystyle\int\limits_{\mathbb{H}^{N}}}
\left\vert \nabla_{\mathbb{H}}f\right\vert ^{2}dV_{\mathbb{H}}-\left(
\frac{N-2}{2}\right)  ^{2}%
{\displaystyle\int\limits_{\mathbb{H}^{N}}}
V_{2}\left(  \frac{e^{\rho\left(  x\right)  }-1}{e^{\rho\left(  x\right)  }%
+1}\right)  \left\vert f\right\vert ^{2}dV_{\mathbb{H}}\\
&  =%
{\displaystyle\int\limits_{\mathbb{H}^{N}}}
\left\vert G\left(  \frac{e^{\rho\left(  x\right)  }-1}{e^{\rho\left(
x\right)  }+1}\right)  \right\vert \left\vert \nabla_{\mathbb{H}}\left(
\frac{f}{\sqrt{G\left(  \frac{e^{\rho\left(  x\right)  }-1}{e^{\rho\left(
x\right)  }+1}\right)  }}\right)  \right\vert ^{2}dV_{\mathbb{H}}.
\end{align*}

\end{corollary}

As a consequence, we obtain the following Hardy type inequality that has been
studied in \cite{ST}:
\[%
{\displaystyle\int\limits_{\mathbb{H}^{N}}}
\left\vert \nabla_{\mathbb{H}}f\right\vert ^{2}dV_{\mathbb{H}}\geq\left(
\frac{N-2}{2}\right)  ^{2}%
{\displaystyle\int\limits_{\mathbb{H}^{N}}}
V_{2}\left(  \frac{e^{\rho\left(  x\right)  }-1}{e^{\rho\left(  x\right)  }%
+1}\right)  \left\vert f\right\vert ^{2}dV_{\mathbb{H}}.
\]

\begin{proof}
We note that $\frac{G}{N\omega_{N-1}}$ is a fundamental solution of the
hyperbolic Laplacian. We have that $\left(  r^{N-1}\frac{1}{\left(
1-r^{2}\right)  ^{N-2}},r^{N-1}\frac{F^{2}\left(  r\right)  }{4G^{2}\left(
r\right)  }\frac{1}{\left(  1-r^{2}\right)  ^{N-2}}\right)  $ is a Bessel pair
on $\left(  0,1\right)  $ with $\varphi=\sqrt{G\left(  r\right)  }$. That is
$\left(  r^{N-1}\frac{1}{\left(  1-r^{2}\right)  ^{N-2}}\varphi^{\prime
}\right)  ^{\prime}+r^{N-1}\frac{F^{2}\left(  r\right)  }{4G^{2}\left(
r\right)  }\frac{1}{\left(  1-r^{2}\right)  ^{N-2}}\varphi=0$. Indeed, a
direct computation shows%
\begin{align*}
r^{N-1}\frac{1}{\left(  1-r^{2}\right)  ^{N-2}}\varphi^{\prime}  &
=r^{N-1}\frac{1}{\left(  1-r^{2}\right)  ^{N-2}}\frac{G^{\prime}\left(
r\right)  }{2\sqrt{G\left(  r\right)  }}=-\frac{r^{N-1}}{\left(
1-r^{2}\right)  ^{N-2}}\frac{\left(  1-r^{2}\right)  ^{N-2}}{r^{N-1}}\frac
{1}{2\sqrt{G\left(  r\right)  }}\\
&  =-\frac{1}{2\sqrt{G\left(  r\right)  }}%
\end{align*}
Thus%
\begin{align*}
\left(  r^{N-1}\frac{1}{\left(  1-r^{2}\right)  ^{N-2}}\varphi^{\prime
}\right)  ^{\prime}  &  =-\left(  \frac{1}{2\sqrt{G\left(  r\right)  }%
}\right)  ^{\prime}\\
&  =\frac{1}{4}\frac{\sqrt{G\left(  r\right)  }G^{\prime}\left(  r\right)
}{G^{2}\left(  r\right)  }\\
&  =-\frac{1}{4}\frac{\sqrt{G\left(  r\right)  }}{G^{2}\left(  r\right)
}\frac{\left(  1-r^{2}\right)  ^{N-2}}{r^{N-1}}\\
&  =-r^{N-1}\frac{F^{2}\left(  r\right)  }{4G^{2}\left(  r\right)  }\frac
{1}{\left(  1-r^{2}\right)  ^{N-2}}\varphi.
\end{align*}
Hence by Theorem \ref{T6}, we obtain
\[%
{\displaystyle\int\limits_{\mathbb{H}^{N}}}
\left\vert \nabla_{\mathbb{H}}f\right\vert ^{2}dV_{\mathbb{H}}=\left(
\frac{N-2}{2}\right)  ^{2}%
{\displaystyle\int\limits_{\mathbb{H}^{N}}}
V_{2}\left\vert f\right\vert ^{2}dV_{\mathbb{H}}+%
{\displaystyle\int\limits_{\mathbb{H}^{N}}}
\left\vert G\left(  \left\vert x\right\vert \right)  \right\vert \left\vert
\nabla_{\mathbb{H}}\left(  \frac{f}{\sqrt{G\left(  \left\vert x\right\vert
\right)  }}\right)  \right\vert ^{2}dV_{\mathbb{H}}.
\]

\end{proof}

Now, we note that $K_{\mathbb{H}^{N}}=-1$. Therefore $J\left(  u,t\right)
=J_{1}\left(  t\right)  $ does not depend on $u$. Moreover%
\[
J_{1}\left(  t\right)  =\left(  \frac{\sinh t}{t}\right)  ^{N-1}.
\]
and
\[
\frac{J^{\prime}\left(  u,t\right)  }{J\left(  u,t\right)  }=\frac
{J_{1}^{\prime}\left(  t\right)  }{J_{1}\left(  t\right)  }=\frac{N-1}%
{t}\left(  t\coth t-1\right)  .
\]
Therefore, we can rewrite Theorem \ref{T1} as follows:

\begin{theorem}
\label{H1}Let $\left(  r^{N-1}V,r^{N-1}W\right)  $ be a Bessel pair on
$\left(  0,R\right)  $. Then we have the following identities%
\begin{align*}
&
{\displaystyle\int\limits_{B_{R}}}
V\left(  \rho\left(  x\right)  \right)  \left\vert \nabla_{\mathbb{H}%
}f\right\vert ^{2}dV_{\mathbb{H}}-%
{\displaystyle\int\limits_{B_{R}}}
W\left(  \rho\left(  x\right)  \right)  \left\vert f\right\vert ^{2}%
dV_{\mathbb{H}}\\
&  =%
{\displaystyle\int\limits_{B_{R}}}
V\left(  \rho\left(  x\right)  \right)  \left\vert \varphi^{2}\left(
\rho\left(  x\right)  \right)  \right\vert \left\vert \nabla_{\mathbb{H}%
}\left(  \frac{f}{\varphi\left(  \rho\left(  x\right)  \right)  }\right)
\right\vert ^{2}dV_{\mathbb{H}}\\
&  -\left(  N-1\right)
{\displaystyle\int\limits_{B_{R}}}
V\left(  \rho\left(  x\right)  \right)  \frac{\varphi^{\prime}\left(
\rho\left(  x\right)  \right)  }{\varphi\left(  \rho\left(  x\right)  \right)
}\frac{\rho\left(  x\right)  \cosh\rho\left(  x\right)  -\sinh\rho\left(
x\right)  }{\rho\left(  x\right)  \sinh\rho\left(  x\right)  }\left\vert
f\right\vert ^{2}dV_{\mathbb{H}}%
\end{align*}
and%
\begin{align*}
&
{\displaystyle\int\limits_{B_{R}}}
V\left(  \rho\left(  x\right)  \right)  \left\vert \partial_{\rho}f\right\vert
^{2}dV_{\mathbb{H}}-%
{\displaystyle\int\limits_{B_{R}}}
W\left(  \rho\left(  x\right)  \right)  \left\vert f\right\vert ^{2}%
dV_{\mathbb{H}}\\
&  =%
{\displaystyle\int\limits_{B_{R}}}
V\left(  \rho\left(  x\right)  \right)  \left\vert \varphi^{2}\left(
\rho\left(  x\right)  \right)  \right\vert \left\vert \partial_{\rho}\left(
\frac{f}{\varphi\left(  \rho\left(  x\right)  \right)  }\right)  \right\vert
^{2}dV_{\mathbb{H}}\\
&  -\left(  N-1\right)
{\displaystyle\int\limits_{B_{R}}}
V\left(  \rho\left(  x\right)  \right)  \frac{\varphi^{\prime}\left(
\rho\left(  x\right)  \right)  }{\varphi\left(  \rho\left(  x\right)  \right)
}\frac{\rho\left(  x\right)  \cosh\rho\left(  x\right)  -\sinh\rho\left(
x\right)  }{\rho\left(  x\right)  \sinh\rho\left(  x\right)  }\left\vert
f\right\vert ^{2}dV_{\mathbb{H}}.
\end{align*}

\end{theorem}

By applying Theorem \ref{H1} to some explicit Bessel pairs, we obtain several
improvements of the Hardy inequalities on hyperbolic spaces.

\begin{proof}
[Proof of Theorem \ref{T3.1}]We apply Theorem \ref{H1} to the Bessel pair
$\left(  r^{N-1},r^{N-1}\left(  \frac{N-2}{2}\right)  ^{2}\frac{1}{r^{2}%
}\right)  $ on $\left(  0,\infty\right)  $. Note that in this case
$\varphi\left(  r\right)  =r^{-\frac{N-2}{2}}$.
\end{proof}

From Theorem \ref{T3.1}, we can deduce the Hardy-Poincar\'{e}--Sobolev
identities and inequalities that provide improved versions with exact
remainder terms of the Hardy-Poincar\'{e}--Sobolev inequalities studied in
\cite{AK13, BGG}.

\begin{proof}
[Proof of Theorem \ref{T3.2}]Let $\Psi\left(  r\right)  =\frac{r}{\sinh r}$
and $\Phi\left(  r\right)  =\Psi\left(  r\right)  ^{\frac{N-1}{2}}=\left(
\frac{r}{\sinh r}\right)  ^{\frac{N-1}{2}}$ and $W\left(  r\right)
=-\frac{\left(  r\Phi^{\prime}\left(  r\right)  \right)  ^{\prime}}%
{r^{N-1}\Phi}$. Then noting that%
\[
\Psi^{\prime}\left(  r\right)  =\frac{\left(  1-r\coth r\right)  }{\sinh r}%
\]
and%
\begin{align*}
\Phi^{\prime}\left(  r\right)   &  =\frac{N-1}{2}\Psi\left(  r\right)
^{\frac{N-3}{2}}\Psi^{\prime}\left(  r\right) \\
&  =\frac{N-1}{2}\left(  \frac{r}{\sinh r}\right)  ^{\frac{N-3}{2}}%
\frac{\left(  1-r\coth r\right)  }{\sinh r}%
\end{align*}
Hence%
\[
\frac{\Phi^{\prime}\left(  r\right)  }{\Phi\left(  r\right)  }=\frac{N-1}%
{2}\frac{1-r\coth r}{r}%
\]
and%
\begin{align*}
r^{N-2}W\left(  r\right)   &  =-\frac{N-1}{2}\frac{\left(  r\left(  \frac
{r}{\sinh r}\right)  ^{\frac{N-3}{2}}\frac{\left(  1-r\coth r\right)  }{\sinh
r}\right)  ^{\prime}}{r\left(  \frac{r}{\sinh r}\right)  ^{\frac{N-1}{2}}}\\
&  =-\frac{N-1}{2}\left[  \frac{\left(  1-r\coth r\right)  }{r^{2}}+\frac
{N-3}{2}\frac{\left(  1-r\coth r\right)  ^{2}}{r^{2}}+\coth^{2}r-2\frac{\coth
r}{r}+\frac{1}{\sinh^{2}r}\right]
\end{align*}
Note that $\left(  r^{N-1}\frac{1}{r^{N-2}}\Phi^{\prime}\left(  r\right)
\right)  ^{\prime}+r^{N-1}W\left(  r\right)  \Phi\left(  r\right)  =0$, we
have by Theorem \ref{H1} that
\begin{align*}
&
{\displaystyle\int\limits_{\mathbb{H}^{N}}}
\frac{1}{\rho^{N-2}\left(  x\right)  }\left\vert \nabla_{\mathbb{H}}\left(
\rho^{\frac{N-2}{2}}\left(  x\right)  f\right)  \right\vert _{\mathbb{H}}%
^{2}dV_{\mathbb{H}}\\
&  =%
{\displaystyle\int\limits_{\mathbb{H}^{N}}}
W\left(  \rho\left(  x\right)  \right)  \left\vert \rho^{\frac{N-2}{2}}\left(
x\right)  f\right\vert ^{2}dV_{\mathbb{H}}\\
&  +%
{\displaystyle\int\limits_{\mathbb{H}^{N}}}
\frac{1}{\rho^{N-2}\left(  x\right)  }\Phi^{2}\left(  \rho\left(  x\right)
\right)  \left\vert \nabla_{\mathbb{H}}\left(  \frac{\rho^{\frac{N-2}{2}%
}\left(  x\right)  f}{\Phi\left(  \rho\left(  x\right)  \right)  }\right)
\right\vert ^{2}dV_{\mathbb{H}}\\
&  +\frac{\left(  N-1\right)  ^{2}}{2}%
{\displaystyle\int\limits_{\mathbb{H}^{N}}}
\frac{1}{\rho^{N-2}\left(  x\right)  }\left(  \frac{\rho\left(  x\right)
\coth\rho\left(  x\right)  -1}{\rho\left(  x\right)  }\right)  ^{2}\left\vert
\rho^{\frac{N-2}{2}}\left(  x\right)  f\right\vert ^{2}dV_{\mathbb{H}}.
\end{align*}
Therefore, from Theorem \ref{T3.1}, we obtain
\begin{align*}
&
{\displaystyle\int\limits_{\mathbb{H}^{N}}}
\left\vert \nabla_{\mathbb{H}}f\right\vert ^{2}dV_{\mathbb{H}}\\
&  =\left(  \frac{N-2}{2}\right)  ^{2}%
{\displaystyle\int\limits_{\mathbb{H}^{N}}}
\frac{\left\vert f\right\vert ^{2}}{\rho^{2}\left(  x\right)  }dV_{\mathbb{H}%
}\\
&  +%
{\displaystyle\int\limits_{\mathbb{H}^{N}}}
\frac{1}{\rho^{N-2}\left(  x\right)  }\left\vert \nabla_{\mathbb{H}}\left(
\rho^{\frac{N-2}{2}}\left(  x\right)  f\right)  \right\vert ^{2}%
dV_{\mathbb{H}}\\
&  +\frac{\left(  N-2\right)  \left(  N-1\right)  }{2}%
{\displaystyle\int\limits_{\mathbb{H}^{N}}}
\frac{\rho\left(  x\right)  \coth\rho\left(  x\right)  -1}{\rho^{2}\left(
x\right)  }\left\vert f\right\vert ^{2}dV_{\mathbb{H}}\\
&  =%
{\displaystyle\int\limits_{\mathbb{H}^{N}}}
\frac{1}{\rho^{N-2}\left(  x\right)  }\Phi^{2}\left(  \rho\left(  x\right)
\right)  \left\vert \nabla_{\mathbb{H}}\left(  \frac{\rho^{\frac{N-2}{2}%
}\left(  x\right)  f}{\Phi\left(  \rho\left(  x\right)  \right)  }\right)
\right\vert ^{2}dV_{\mathbb{H}}\\
&  +\left(  \frac{N-2}{2}\right)  ^{2}%
{\displaystyle\int\limits_{\mathbb{H}^{N}}}
\frac{\left\vert f\right\vert ^{2}}{\rho^{2}\left(  x\right)  }dV_{\mathbb{H}%
}+\frac{\left(  N-2\right)  \left(  N-1\right)  }{2}%
{\displaystyle\int\limits_{\mathbb{H}^{N}}}
\frac{\rho\left(  x\right)  \coth\rho\left(  x\right)  -1}{\rho^{2}\left(
x\right)  }\left\vert f\right\vert ^{2}dV_{\mathbb{H}}\\
&  +\frac{\left(  N-1\right)  ^{2}}{2}%
{\displaystyle\int\limits_{\mathbb{H}^{N}}}
\left(  \frac{\rho\left(  x\right)  \coth\rho\left(  x\right)  -1}{\rho\left(
x\right)  }\right)  ^{2}\left\vert f\right\vert ^{2}+%
{\displaystyle\int\limits_{\mathbb{H}^{N}}}
\rho^{N-2}\left(  x\right)  W\left(  \rho\left(  x\right)  \right)  \left\vert
f\right\vert ^{2}dV_{\mathbb{H}}\\
&  =%
{\displaystyle\int\limits_{\mathbb{H}^{N}}}
\frac{1}{\rho^{N-2}\left(  x\right)  }\Phi^{2}\left(  \rho\left(  x\right)
\right)  \left\vert \nabla_{\mathbb{H}}\left(  \frac{\rho^{\frac{N-2}{2}%
}\left(  x\right)  f}{\Phi\left(  \rho\left(  x\right)  \right)  }\right)
\right\vert ^{2}dV_{\mathbb{H}}\\
&  +%
{\displaystyle\int\limits_{\mathbb{H}^{N}}}
\left[
\begin{array}
[c]{c}%
\left(  \frac{N-2}{2}\right)  ^{2}\frac{1}{\rho^{2}\left(  x\right)  }%
+\frac{\left(  N-2\right)  \left(  N-1\right)  }{2}\frac{\rho\left(  x\right)
\coth\rho\left(  x\right)  -1}{\rho^{2}\left(  x\right)  }+\frac{\left(
N-1\right)  ^{2}}{2}\left(  \frac{\rho\left(  x\right)  \coth\rho\left(
x\right)  -1}{\rho\left(  x\right)  }\right)  ^{2}\\
-\frac{N-1}{2}\left[  \frac{\left(  1-\rho\left(  x\right)  \coth\rho\left(
x\right)  \right)  }{\rho^{2}\left(  x\right)  }+\frac{N-3}{2}\frac{\left(
1-\rho\left(  x\right)  \coth\rho\left(  x\right)  \right)  ^{2}}{\rho
^{2}\left(  x\right)  }+\coth^{2}\rho\left(  x\right)  -2\frac{\coth
\rho\left(  x\right)  }{\rho\left(  x\right)  }+\frac{1}{\sinh^{2}r}\right]
\end{array}
\right]  \left\vert f\right\vert ^{2}dV_{\mathbb{H}}\\
&  =%
{\displaystyle\int\limits_{\mathbb{H}^{N}}}
\frac{1}{\rho^{N-2}\left(  x\right)  }\Phi^{2}\left(  \rho\left(  x\right)
\right)  \left\vert \nabla_{\mathbb{H}}\left(  \frac{\rho^{\frac{N-2}{2}%
}\left(  x\right)  f}{\Phi\left(  \rho\left(  x\right)  \right)  }\right)
\right\vert ^{2}dV_{\mathbb{H}}\\
&  +%
{\displaystyle\int\limits_{\mathbb{H}^{N}}}
\left[  \frac{\left(  N-1\right)  ^{2}}{4}+\frac{1}{4}\frac{1}{\rho^{2}\left(
x\right)  }+\frac{\left(  N-1\right)  \left(  N-3\right)  }{4\sinh^{2}%
\rho\left(  x\right)  }\right]  \left\vert f\right\vert ^{2}dV_{\mathbb{H}}%
\end{align*}
In other words,%
\begin{align*}
&
{\displaystyle\int\limits_{\mathbb{H}^{N}}}
\left\vert \nabla_{\mathbb{H}}f\right\vert ^{2}dV_{\mathbb{H}}-%
{\displaystyle\int\limits_{\mathbb{H}^{N}}}
\left[  \frac{\left(  N-1\right)  ^{2}}{4}+\frac{1}{4}\frac{1}{\rho^{2}\left(
x\right)  }+\frac{\left(  N-1\right)  \left(  N-3\right)  }{4}\frac{1}%
{\sinh^{2}\rho\left(  x\right)  }\right]  \left\vert f\right\vert
^{2}dV_{\mathbb{H}}\\
&  =%
{\displaystyle\int\limits_{\mathbb{H}^{N}}}
\frac{1}{\rho^{N-2}\left(  x\right)  }\Phi^{2}\left(  \rho\left(  x\right)
\right)  \left\vert \nabla_{\mathbb{H}}\left(  \frac{\rho^{\frac{N-2}{2}%
}\left(  x\right)  f}{\Phi\left(  \rho\left(  x\right)  \right)  }\right)
\right\vert ^{2}dV_{\mathbb{H}}.
\end{align*}
Similarly, we also get%
\begin{align*}
&
{\displaystyle\int\limits_{\mathbb{H}^{N}}}
\left\vert \partial_{\rho}f\right\vert ^{2}dV_{\mathbb{H}}-%
{\displaystyle\int\limits_{\mathbb{H}^{N}}}
\left[  \frac{\left(  N-1\right)  ^{2}}{4}+\frac{1}{4}\frac{1}{\rho^{2}\left(
x\right)  }+\frac{\left(  N-1\right)  \left(  N-3\right)  }{4}\frac{1}%
{\sinh^{2}\rho\left(  x\right)  }\right]  \left\vert f\right\vert
^{2}dV_{\mathbb{H}}\\
&  =%
{\displaystyle\int\limits_{\mathbb{H}^{N}}}
\frac{1}{\rho^{N-2}\left(  x\right)  }\Phi^{2}\left(  \rho\left(  x\right)
\right)  \left\vert \partial_{\rho}\left(  \frac{\rho^{\frac{N-2}{2}}\left(
x\right)  f}{\Phi\left(  \rho\left(  x\right)  \right)  }\right)  \right\vert
^{2}dV_{\mathbb{H}}.
\end{align*}

\end{proof}

\begin{corollary}
We have%
\begin{align*}
&
{\displaystyle\int\limits_{\mathbb{H}^{N}}}
\frac{\left\vert \nabla_{\mathbb{H}}f\right\vert ^{2}}{\rho^{\lambda}\left(
x\right)  }dV_{\mathbb{H}}-\frac{\left(  N-\lambda-2\right)  ^{2}}{4}%
{\displaystyle\int\limits_{\mathbb{H}^{N}}}
\frac{\left\vert f\right\vert ^{2}}{\rho^{\lambda+2}\left(  x\right)
}dV_{\mathbb{H}}\\
&  =%
{\displaystyle\int\limits_{\mathbb{H}^{N}}}
\frac{1}{\rho^{N-2}\left(  x\right)  }\left\vert \nabla_{\mathbb{H}}\left(
\rho^{\frac{N-\lambda-2}{2}}\left(  x\right)  f\right)  \right\vert
^{2}dV_{\mathbb{H}}\\
&  +\frac{\left(  N-\lambda-2\right)  \left(  N-1\right)  }{2}%
{\displaystyle\int\limits_{\mathbb{H}^{N}}}
\frac{\rho\left(  x\right)  \cosh\rho\left(  x\right)  -\sinh\rho\left(
x\right)  }{\rho^{\lambda+2}\left(  x\right)  \sinh\rho\left(  x\right)
}\left\vert f\right\vert ^{2}dV_{\mathbb{H}}%
\end{align*}
and%
\begin{align*}
&
{\displaystyle\int\limits_{\mathbb{H}^{N}}}
\frac{\left\vert \partial_{\rho}f\right\vert ^{2}}{\rho^{\lambda}\left(
x\right)  }dV_{\mathbb{H}}-\frac{\left(  N-\lambda-2\right)  ^{2}}{4}%
{\displaystyle\int\limits_{\mathbb{H}^{N}}}
\frac{\left\vert f\right\vert ^{2}}{\rho^{\lambda+2}\left(  x\right)
}dV_{\mathbb{H}}\\
&  =%
{\displaystyle\int\limits_{\mathbb{H}^{N}}}
\frac{1}{\rho^{N-2}\left(  x\right)  }\left\vert \partial_{\rho}\left(
\rho^{\frac{N-\lambda-2}{2}}\left(  x\right)  f\right)  \right\vert
^{2}dV_{\mathbb{H}}\\
&  +\frac{\left(  N-\lambda-2\right)  \left(  N-1\right)  }{2}%
{\displaystyle\int\limits_{\mathbb{H}^{N}}}
\frac{\rho\left(  x\right)  \cosh\rho\left(  x\right)  -\sinh\rho\left(
x\right)  }{\rho^{\lambda+2}\left(  x\right)  \sinh\rho\left(  x\right)
}\left\vert f\right\vert ^{2}dV_{\mathbb{H}}.
\end{align*}
As a consequence of these identities, we get that for $\lambda<N-2:$
\begin{align*}
&
{\displaystyle\int\limits_{\mathbb{H}^{N}}}
\frac{\left\vert \nabla_{\mathbb{H}}f\right\vert ^{2}}{\rho^{\lambda}\left(
x\right)  }dV_{\mathbb{H}}\\
&  \geq%
{\displaystyle\int\limits_{\mathbb{H}^{N}}}
\frac{\left\vert \partial_{\rho}f\right\vert ^{2}}{\rho^{\lambda}\left(
x\right)  }dV_{\mathbb{H}}\\
&  \geq\frac{\left(  N-\lambda-2\right)  ^{2}}{4}%
{\displaystyle\int\limits_{\mathbb{H}^{N}}}
\frac{\left\vert f\right\vert ^{2}}{\rho^{\lambda+2}\left(  x\right)
}dV_{\mathbb{H}}\\
&  +\frac{\left(  N-\lambda-2\right)  \left(  N-1\right)  }{2}%
{\displaystyle\int\limits_{\mathbb{H}^{N}}}
\frac{\rho\left(  x\right)  \cosh\rho\left(  x\right)  -\sinh\rho\left(
x\right)  }{\rho^{\lambda+2}\left(  x\right)  \sinh\rho\left(  x\right)
}\left\vert f\right\vert ^{2}dV_{\mathbb{H}}\\
&  \geq\frac{\left(  N-\lambda-2\right)  ^{2}}{4}%
{\displaystyle\int\limits_{\mathbb{H}^{N}}}
\frac{\left\vert f\right\vert ^{2}}{\rho^{\lambda+2}\left(  x\right)
}dV_{\mathbb{H}}.
\end{align*}

\end{corollary}

\begin{proof}
$\left(  r^{N-1}r^{-\lambda},r^{N-1}r^{-\lambda}\frac{\left(  N-\lambda
-2\right)  ^{2}}{4}\frac{1}{r^{2}}\right)  $ is a Bessel pair on $\left(
0,\infty\right)  $ with $\varphi\left(  r\right)  =r^{\frac{2-N+\lambda}{2}}$.
\end{proof}

\begin{corollary}
We have%
\begin{align*}
&
{\displaystyle\int\limits_{0<\rho\left(  x\right)  <R}}
\frac{\left\vert \nabla_{\mathbb{H}}f\right\vert ^{2}}{\rho^{N-2}\left(
x\right)  }dV_{\mathbb{H}}-\frac{z_{0}^{2}}{R^{2}}%
{\displaystyle\int\limits_{0<\rho\left(  x\right)  <R}}
\frac{\left\vert f\right\vert ^{2}}{\rho^{N-2}\left(  x\right)  }%
dV_{\mathbb{H}}\\
&  =%
{\displaystyle\int\limits_{0<\rho\left(  x\right)  <R}}
\frac{J_{0}^{2}\left(  \frac{z_{0}}{R}\rho\left(  x\right)  \right)  }%
{\rho^{N-2}\left(  x\right)  }\left\vert \nabla_{\mathbb{H}}\left(  \frac
{f}{J_{0}\left(  \frac{z_{0}}{R}\rho\left(  x\right)  \right)  }\right)
\right\vert ^{2}dV_{\mathbb{H}}\\
&  -\left(  N-1\right)  \frac{z_{0}}{R}%
{\displaystyle\int\limits_{0<\rho\left(  x\right)  <R}}
\frac{J_{0}^{\prime}\left(  \frac{z_{0}}{R}\rho\left(  x\right)  \right)
}{J_{0}\left(  \frac{z_{0}}{R}\rho\left(  x\right)  \right)  }\frac
{\rho\left(  x\right)  \cosh\rho\left(  x\right)  -\sinh\rho\left(  x\right)
}{\rho^{N-1}\left(  x\right)  \sinh\rho\left(  x\right)  }\left\vert
f\right\vert ^{2}dV_{\mathbb{H}}%
\end{align*}
and
\begin{align*}
&
{\displaystyle\int\limits_{0<\rho\left(  x\right)  <R}}
\frac{\left\vert \partial_{\rho}f\right\vert ^{2}}{\rho^{N-2}\left(  x\right)
}dV_{\mathbb{H}}-\frac{z_{0}^{2}}{R^{2}}%
{\displaystyle\int\limits_{0<\rho\left(  x\right)  <R}}
\frac{\left\vert f\right\vert ^{2}}{\rho^{N-2}\left(  x\right)  }%
dV_{\mathbb{H}}\\
&  =%
{\displaystyle\int\limits_{0<\rho\left(  x\right)  <R}}
\frac{J_{0}^{2}\left(  \frac{z_{0}}{R}\rho\left(  x\right)  \right)  }%
{\rho^{N-2}\left(  x\right)  }\left\vert \partial_{\rho}\left(  \frac{f}%
{J_{0}\left(  \frac{z_{0}}{R}\rho\left(  x\right)  \right)  }\right)
\right\vert ^{2}dV_{\mathbb{H}}\\
&  -\left(  N-1\right)  \frac{z_{0}}{R}%
{\displaystyle\int\limits_{0<\rho\left(  x\right)  <R}}
\frac{J_{0}^{\prime}\left(  \frac{z_{0}}{R}\rho\left(  x\right)  \right)
}{J_{0}\left(  \frac{z_{0}}{R}\rho\left(  x\right)  \right)  }\frac
{\rho\left(  x\right)  \cosh\rho\left(  x\right)  -\sinh\rho\left(  x\right)
}{\rho^{N-1}\left(  x\right)  \sinh\rho\left(  x\right)  }\left\vert
f\right\vert ^{2}dV_{\mathbb{H}}.
\end{align*}
As a consequence of these identities, we get that%
\begin{align*}
&
{\displaystyle\int\limits_{0<\rho\left(  x\right)  <R}}
\frac{\left\vert \nabla_{\mathbb{H}}f\right\vert ^{2}}{\rho^{N-2}\left(
x\right)  }dV_{\mathbb{H}}\\
&  \geq%
{\displaystyle\int\limits_{0<\rho\left(  x\right)  <R}}
\frac{\left\vert \partial_{\rho}f\right\vert ^{2}}{\rho^{N-2}\left(  x\right)
}dV_{\mathbb{H}}\\
&  \geq\frac{z_{0}^{2}}{R^{2}}%
{\displaystyle\int\limits_{0<\rho\left(  x\right)  <R}}
\frac{\left\vert f\right\vert ^{2}}{\rho^{N-2}\left(  x\right)  }%
dV_{\mathbb{H}}\\
&  -\left(  N-1\right)  \frac{z_{0}}{R}%
{\displaystyle\int\limits_{0<\rho\left(  x\right)  <R}}
\frac{J_{0}^{\prime}\left(  \frac{z_{0}}{R}\rho\left(  x\right)  \right)
}{J_{0}\left(  \frac{z_{0}}{R}\rho\left(  x\right)  \right)  }\frac
{\rho\left(  x\right)  \cosh\rho\left(  x\right)  -\sinh\rho\left(  x\right)
}{\rho^{N-1}\left(  x\right)  \sinh\rho\left(  x\right)  }\left\vert
f\right\vert ^{2}dV_{\mathbb{H}}\\
&  \geq\frac{z_{0}^{2}}{R^{2}}%
{\displaystyle\int\limits_{0<\rho\left(  x\right)  <R}}
\frac{\left\vert f\right\vert ^{2}}{\rho^{N-2}\left(  x\right)  }%
dV_{\mathbb{H}}.
\end{align*}

\end{corollary}

\begin{proof}
$\left(  r^{N-1}r^{2-N},r^{N-1}r^{2-N}\frac{z_{0}^{2}}{R^{2}}\right)  $ is a
Bessel pair on $\left(  0,R\right)  $ with $\varphi\left(  r\right)
=J_{0}\left(  \frac{z_{0}}{R}r\right)  $ and $\varphi^{\prime}\left(
r\right)  =\frac{z_{0}}{R}J_{0}^{\prime}\left(  \frac{z_{0}}{R}r\right)  $.
\end{proof}

\begin{proof}
[Proof of Theorem \ref{T3.3}]We note that $\left(  r^{N-1}r^{-\lambda}%
,r^{N-1}r^{-\lambda}\left[  \left(  \frac{\left(  N-\lambda-2\right)  ^{2}}%
{4}-\alpha^{2}\right)  \frac{1}{r^{2}}+\frac{z_{\alpha}^{2}}{R^{2}}\right]
\right)  $ on $\left(  0,R\right)  $ with $\varphi\left(  r\right)
=r^{\frac{2-N+\lambda}{2}}J_{\alpha}\left(  \frac{z_{\alpha}}{R}r\right)
,0\leq\alpha\leq\frac{N-\lambda-2}{2}$. Here $z_{\alpha}$ is the first zero of
the Bessel function $J_{\alpha}\left(  z\right)  $. Now, we can apply Theorem
\ref{H1} to obtain the desired results.
\end{proof}

\begin{corollary}
We have%
\begin{align*}
&
{\displaystyle\int\limits_{0<\rho\left(  x\right)  <R}}
\frac{\left\vert \nabla_{\mathbb{H}}f\right\vert ^{2}}{\rho^{N-2}\left(
x\right)  }dV_{\mathbb{H}}-\frac{1}{4}%
{\displaystyle\int\limits_{0<\rho\left(  x\right)  <R}}
\frac{\left\vert f\right\vert ^{2}}{\rho^{N}\left(  x\right)  \left\vert
\ln\frac{\rho\left(  x\right)  }{R}\right\vert ^{2}}dV_{\mathbb{H}}\\
&  =%
{\displaystyle\int\limits_{0<\rho\left(  x\right)  <R}}
\frac{1}{\rho^{N-2}\left(  x\right)  }\left\vert \ln\frac{\rho\left(
x\right)  }{R}\right\vert \left\vert \nabla_{\mathbb{H}}\left(  \frac{f}%
{\sqrt{\left\vert \ln\frac{\left\vert x\right\vert }{R}\right\vert }}\right)
\right\vert ^{2}dV_{\mathbb{H}}\\
&  +%
{\displaystyle\int\limits_{0<\rho\left(  x\right)  <R}}
\frac{1}{\rho^{N-2}\left(  x\right)  }\frac{1}{2\left\vert \ln\frac
{\rho\left(  x\right)  }{R}\right\vert }\frac{\rho\left(  x\right)  \cosh
\rho\left(  x\right)  -\sinh\rho\left(  x\right)  }{\rho^{2}\left(  x\right)
\sinh\rho\left(  x\right)  }\left\vert f\right\vert ^{2}dV_{\mathbb{H}}%
\end{align*}
and%
\begin{align*}
&
{\displaystyle\int\limits_{0<\rho\left(  x\right)  <R}}
\frac{\left\vert \partial_{\rho}f\right\vert ^{2}}{\rho^{N-2}\left(  x\right)
}dV_{\mathbb{H}}-\frac{1}{4}%
{\displaystyle\int\limits_{0<\rho\left(  x\right)  <R}}
\frac{\left\vert f\right\vert ^{2}}{\rho^{N}\left(  x\right)  \left\vert
\ln\frac{\rho\left(  x\right)  }{R}\right\vert ^{2}}dV_{\mathbb{H}}\\
&  =%
{\displaystyle\int\limits_{0<\rho\left(  x\right)  <R}}
\frac{1}{\rho^{N-2}\left(  x\right)  }\left\vert \ln\frac{\rho\left(
x\right)  }{R}\right\vert \left\vert \partial_{\rho}\left(  \frac{f}%
{\sqrt{\left\vert \ln\frac{\left\vert x\right\vert }{R}\right\vert }}\right)
\right\vert ^{2}dV_{\mathbb{H}}\\
&  +%
{\displaystyle\int\limits_{0<\rho\left(  x\right)  <R}}
\frac{1}{\rho^{N-2}\left(  x\right)  }\frac{1}{2\left\vert \ln\frac
{\rho\left(  x\right)  }{R}\right\vert }\frac{\rho\left(  x\right)  \cosh
\rho\left(  x\right)  -\sinh\rho\left(  x\right)  }{\rho^{2}\left(  x\right)
\sinh\rho\left(  x\right)  }\left\vert f\right\vert ^{2}dV_{\mathbb{H}}%
\end{align*}
As a consequence of these identities, we get that%
\begin{align*}
&
{\displaystyle\int\limits_{0<\rho\left(  x\right)  <R}}
\frac{\left\vert \nabla_{\mathbb{H}}f\right\vert ^{2}}{\rho^{N-2}\left(
x\right)  }dV_{\mathbb{H}}\\
&  \geq%
{\displaystyle\int\limits_{0<\rho\left(  x\right)  <R}}
\frac{\left\vert \partial_{\rho}f\right\vert ^{2}}{\rho^{N-2}\left(  x\right)
}dV_{\mathbb{H}}\\
&  \geq\frac{1}{4}%
{\displaystyle\int\limits_{0<\rho\left(  x\right)  <R}}
\frac{\left\vert f\right\vert ^{2}}{\rho^{N}\left(  x\right)  \left\vert
\ln\frac{\rho\left(  x\right)  }{R}\right\vert ^{2}}dV_{\mathbb{H}}\\
&  +%
{\displaystyle\int\limits_{0<\rho\left(  x\right)  <R}}
\frac{1}{\rho^{N-2}\left(  x\right)  }\frac{1}{2\left\vert \ln\frac
{\rho\left(  x\right)  }{R}\right\vert }\frac{\rho\left(  x\right)  \cosh
\rho\left(  x\right)  -\sinh\rho\left(  x\right)  }{\rho^{2}\left(  x\right)
\sinh\rho\left(  x\right)  }\left\vert f\right\vert ^{2}dV_{\mathbb{H}}\\
&  \geq\frac{1}{4}%
{\displaystyle\int\limits_{0<\rho\left(  x\right)  <R}}
\frac{\left\vert f\right\vert ^{2}}{\rho^{N}\left(  x\right)  \left\vert
\ln\frac{\rho\left(  x\right)  }{R}\right\vert ^{2}}dV_{\mathbb{H}}.
\end{align*}

\end{corollary}

\begin{proof}
$\left(  r^{N-1}\frac{1}{r^{N-2}},r^{N-1}\frac{1}{4r^{N}\left\vert \ln\frac
{r}{R}\right\vert ^{2}}\right)  $ is a Bessel pair on $\left(  0,R\right)  $
with $\varphi\left(  r\right)  =\sqrt{\left\vert \ln\frac{r}{R}\right\vert }$
and $\varphi^{\prime}\left(  r\right)  =-\frac{1}{2r\sqrt{\left\vert \ln
\frac{r}{R}\right\vert }}$.
\end{proof}

\section{Proofs of Theorem \ref{T1} and Theorem \ref{T2}}

\begin{proof}
[Proof of Theorem \ref{T1}]Let $f\left(  x\right)  =\varphi\left(  \rho\left(
x\right)  \right)  v\left(  x\right)  $, then
\begin{align*}
&
{\displaystyle\int\limits_{B_{R}\left(  O\right)  }}
V\left(  \rho\left(  x\right)  \right)  \left\vert \nabla_{g}f\right\vert
_{g}^{2}dV_{g}\\
&  =%
{\displaystyle\int\limits_{B_{R}\left(  O\right)  }}
V\left(  \rho\left(  x\right)  \right)  \left\vert \nabla_{g}\left(
\varphi\left(  \rho\left(  x\right)  \right)  v\left(  x\right)  \right)
\right\vert _{g}^{2}dV_{g}\\
&  =%
{\displaystyle\int\limits_{B_{R}\left(  O\right)  }}
V\left(  \rho\left(  x\right)  \right)  \left\vert \varphi^{2}\left(
\rho\left(  x\right)  \right)  \right\vert \left\vert \nabla_{g}v\right\vert
_{g}^{2}dV_{g}+%
{\displaystyle\int\limits_{B_{R}\left(  O\right)  }}
V\left(  \rho\left(  x\right)  \right)  \left\vert v^{2}\left(  x\right)
\right\vert \left\vert \nabla_{g}\varphi\left(  \rho\left(  x\right)  \right)
\right\vert _{g}^{2}dV_{g}\\
&  +%
{\displaystyle\int\limits_{B_{R}\left(  O\right)  }}
V\left(  \rho\left(  x\right)  \right)  \varphi\left(  \rho\left(  x\right)
\right)  \left\langle \nabla_{g}v^{2},\nabla_{g}\varphi\left(  \rho\left(
x\right)  \right)  \right\rangle _{g}dV_{g}\\
&  =%
{\displaystyle\int\limits_{B_{R}\left(  O\right)  }}
V\left(  \rho\left(  x\right)  \right)  \left\vert \varphi^{2}\left(
\rho\left(  x\right)  \right)  \right\vert \left\vert \nabla_{g}v\right\vert
_{g}^{2}dV_{g}+%
{\displaystyle\int\limits_{B_{R}\left(  O\right)  }}
V\left(  \rho\left(  x\right)  \right)  \left\vert v^{2}\left(  x\right)
\right\vert \left\vert \varphi^{\prime}\left(  \rho\left(  x\right)  \right)
\right\vert ^{2}dV_{g}\\
&  +%
{\displaystyle\int\limits_{B_{R}\left(  O\right)  }}
V\left(  \rho\left(  x\right)  \right)  \varphi\left(  \rho\left(  x\right)
\right)  \varphi^{\prime}\left(  \rho\left(  x\right)  \right)  \left\langle
\nabla_{g}v^{2},\nabla_{g}\rho\left(  x\right)  \right\rangle _{g}dV_{g}%
\end{align*}
Now, using the divergence theorem, we get%
\begin{align*}
&
{\displaystyle\int\limits_{B_{R}\left(  O\right)  }}
V\left(  \rho\left(  x\right)  \right)  \varphi\left(  \rho\left(  x\right)
\right)  \varphi^{\prime}\left(  \rho\left(  x\right)  \right)  \left\langle
\nabla_{g}v^{2},\nabla_{g}\rho\left(  x\right)  \right\rangle _{g}dV_{g}\\
&  =-%
{\displaystyle\int\limits_{B_{R}\left(  O\right)  }}
v^{2}\left(  x\right)  \operatorname{div}\left(  V\left(  \rho\left(
x\right)  \right)  \varphi\left(  \rho\left(  x\right)  \right)
\varphi^{\prime}\left(  \rho\left(  x\right)  \right)  \nabla_{g}\rho\left(
x\right)  \right) \\
&  =-%
{\displaystyle\int\limits_{B_{R}\left(  O\right)  }}
v^{2}\left(  x\right)  V\left(  \rho\left(  x\right)  \right)  \varphi\left(
\rho\left(  x\right)  \right)  \varphi^{\prime}\left(  \rho\left(  x\right)
\right)  \Delta_{g}\rho\left(  x\right)  dV_{g}\\
&  -%
{\displaystyle\int\limits_{B_{R}\left(  O\right)  }}
v^{2}\left(  x\right)  V^{\prime}\left(  \rho\left(  x\right)  \right)
\varphi\left(  \rho\left(  x\right)  \right)  \varphi^{\prime}\left(
\rho\left(  x\right)  \right)  dV_{g}\\
&  -%
{\displaystyle\int\limits_{B_{R}\left(  O\right)  }}
v^{2}\left(  x\right)  V\left(  \rho\left(  x\right)  \right)  \varphi
^{\prime}\left(  \rho\left(  x\right)  \right)  \varphi^{\prime}\left(
\rho\left(  x\right)  \right)  dV_{g}\\
&  -%
{\displaystyle\int\limits_{B_{R}\left(  O\right)  }}
v^{2}\left(  x\right)  V\left(  \rho\left(  x\right)  \right)  \varphi\left(
\rho\left(  x\right)  \right)  \varphi^{\prime\prime}\left(  \rho\left(
x\right)  \right)  dV_{g}%
\end{align*}
Noting that (see \cite[4.B.2]{GHL})
\[
\Delta_{g}\rho\left(  x\right)  =\frac{N-1}{\rho\left(  x\right)  }%
+\frac{J^{\prime}\left(  u,\rho\left(  x\right)  \right)  }{J\left(
u,\rho\left(  x\right)  \right)  }.
\]
Hence%
\begin{align*}
&
{\displaystyle\int\limits_{B_{R}\left(  O\right)  }}
V\left(  \rho\left(  x\right)  \right)  \left\vert \nabla_{g}f\right\vert
_{g}^{2}dV_{g}-%
{\displaystyle\int\limits_{B_{R}\left(  O\right)  }}
V\left(  \rho\left(  x\right)  \right)  \left\vert \varphi^{2}\left(
\rho\left(  x\right)  \right)  \right\vert \left\vert \nabla_{g}v\right\vert
_{g}^{2}dV_{g}\\
&  =-%
{\displaystyle\int\limits_{B_{R}\left(  O\right)  }}
\varphi\left(  \rho\left(  x\right)  \right)  v^{2}\left(  x\right)  \left[
\begin{array}
[c]{c}%
V\left(  \rho\left(  x\right)  \right)  \varphi^{\prime}\left(  \rho\left(
x\right)  \right)  \frac{N-1}{\rho\left(  x\right)  }\\
+V^{\prime}\left(  \rho\left(  x\right)  \right)  \varphi^{\prime}\left(
\rho\left(  x\right)  \right)  +V\left(  \rho\left(  x\right)  \right)
\varphi^{\prime\prime}\left(  \rho\left(  x\right)  \right)
\end{array}
\right]  dV_{g}\\
&  -%
{\displaystyle\int\limits_{B_{R}\left(  O\right)  }}
v^{2}\left(  x\right)  V\left(  \rho\left(  x\right)  \right)  \varphi\left(
\rho\left(  x\right)  \right)  \varphi^{\prime}\left(  \rho\left(  x\right)
\right)  \frac{J^{\prime}\left(  u,\rho\left(  x\right)  \right)  }{J\left(
u,\rho\left(  x\right)  \right)  }dV_{g}\\
&  =%
{\displaystyle\int\limits_{B_{R}\left(  O\right)  }}
\varphi^{2}\left(  \rho\left(  x\right)  \right)  v^{2}\left(  x\right)
W\left(  \rho\left(  x\right)  \right)  -%
{\displaystyle\int\limits_{B_{R}\left(  O\right)  }}
v^{2}\left(  x\right)  V\left(  \rho\left(  x\right)  \right)  \varphi\left(
\rho\left(  x\right)  \right)  \varphi^{\prime}\left(  \rho\left(  x\right)
\right)  \frac{J^{\prime}\left(  u,\rho\left(  x\right)  \right)  }{J\left(
u,\rho\left(  x\right)  \right)  }dV_{g}\\
&  =%
{\displaystyle\int\limits_{B_{R}\left(  O\right)  }}
W\left(  \rho\left(  x\right)  \right)  \left\vert f\right\vert ^{2}dV_{g}-%
{\displaystyle\int\limits_{B_{R}\left(  O\right)  }}
v^{2}\left(  x\right)  V\left(  \rho\left(  x\right)  \right)  \varphi\left(
\rho\left(  x\right)  \right)  \varphi^{\prime}\left(  \rho\left(  x\right)
\right)  \frac{J^{\prime}\left(  u,\rho\left(  x\right)  \right)  }{J\left(
u,\rho\left(  x\right)  \right)  }dV_{g}\\
&  =%
{\displaystyle\int\limits_{B_{R}\left(  O\right)  }}
W\left(  \rho\left(  x\right)  \right)  \left\vert f\right\vert ^{2}dV_{g}-%
{\displaystyle\int\limits_{B_{R}\left(  O\right)  }}
V\left(  \rho\left(  x\right)  \right)  \left\vert f\right\vert ^{2}%
\frac{\varphi^{\prime}\left(  \rho\left(  x\right)  \right)  }{\varphi\left(
\rho\left(  x\right)  \right)  }\frac{J^{\prime}\left(  u,\rho\left(
x\right)  \right)  }{J\left(  u,\rho\left(  x\right)  \right)  }dV_{g}.
\end{align*}
Now, denote $F\left(  y\right)  =f\left(  \exp_{O}\left(  y\right)  \right)
$, $\Phi\left(  y\right)  =\varphi\left(  \exp_{O}\left(  y\right)  \right)  $
and $\Psi\left(  y\right)  =v\left(  \exp_{O}\left(  y\right)  \right)  $.
Then using the polar coordinate we get%
\begin{align*}
&
{\displaystyle\int\limits_{B_{R}\left(  O\right)  }}
W\left(  \rho\left(  x\right)  \right)  \left\vert f\right\vert ^{2}dV_{g}\\
&  =%
{\displaystyle\int\limits_{\mathbb{S}^{N-1}}}
{\displaystyle\int\limits_{0}^{R}}
\rho^{N-1}W\left(  \rho\right)  \Phi\left(  \rho\right)  \Phi\left(
\rho\right)  \Psi^{2}\left(  \rho u\right)  J\left(  u,\rho\right)  d\rho du\\
&  =-%
{\displaystyle\int\limits_{\mathbb{S}^{N-1}}}
{\displaystyle\int\limits_{0}^{R}}
\partial_{\rho}\left(  \rho^{N-1}V\left(  \rho\right)  \partial_{\rho}%
\Phi\left(  \rho\right)  \right)  \Phi\left(  \rho\right)  \Psi^{2}\left(
\rho u\right)  J\left(  u,\rho\right)  d\rho du\\
&  =%
{\displaystyle\int\limits_{\mathbb{S}^{N-1}}}
{\displaystyle\int\limits_{0}^{R}}
\rho^{N-1}V\left(  \rho\right)  \partial_{\rho}\Phi\left(  \rho\right)
\partial_{\rho}\left[  \Phi\left(  \rho\right)  \Psi^{2}\left(  \rho u\right)
J\left(  u,\rho\right)  \right]  d\rho du\\
&  =%
{\displaystyle\int\limits_{\mathbb{S}^{N-1}}}
{\displaystyle\int\limits_{0}^{R}}
\rho^{N-1}V\left(  \rho\right)  \left(  \partial_{\rho}\Phi\left(
\rho\right)  \right)  ^{2}\Psi^{2}\left(  \rho u\right)  J\left(
u,\rho\right)  d\rho du\\
&  +2%
{\displaystyle\int\limits_{\mathbb{S}^{N-1}}}
{\displaystyle\int\limits_{0}^{R}}
\rho^{N-1}V\left(  \rho\right)  \partial_{\rho}\Phi\left(  \rho\right)
\Phi\left(  \rho\right)  \partial_{\rho}\Psi\left(  \rho u\right)  \Psi\left(
\rho u\right)  J\left(  u,\rho\right)  d\rho du\\
&  +%
{\displaystyle\int\limits_{\mathbb{S}^{N-1}}}
{\displaystyle\int\limits_{0}^{R}}
\rho^{N-1}V\left(  \rho\right)  \partial_{\rho}\Phi\left(  \rho\right)
\Phi\left(  \rho\right)  \Psi^{2}\left(  \rho u\right)  \partial_{\rho
}J\left(  u,\rho\right)  d\rho du.
\end{align*}

Hence, we have%
\begin{align*}
&
{\displaystyle\int\limits_{B_{R}\left(  O\right)  }}
W\left(  \rho\left(  x\right)  \right)  \left\vert f\right\vert ^{2}dV_{g}\\
&  =%
{\displaystyle\int\limits_{\mathbb{S}^{N-1}}}
{\displaystyle\int\limits_{0}^{R}}
\rho^{N-1}V\left(  \rho\right)  \left\vert \Psi\left(  \rho u\right)
\partial_{\rho}\Phi\left(  \rho\right)  +\Phi\left(  \rho\right)
\partial_{\rho}\Psi\left(  \rho u\right)  \right\vert ^{2}J\left(
u,\rho\right)  d\rho du\\
&  -%
{\displaystyle\int\limits_{\mathbb{S}^{N-1}}}
{\displaystyle\int\limits_{0}^{R}}
\rho^{N-1}V\left(  \rho\right)  \left\vert \Phi\left(  \rho\right)
\partial_{\rho}\Psi\left(  \rho u\right)  \right\vert ^{2}J\left(
u,\rho\right)  d\rho du\\
&  +%
{\displaystyle\int\limits_{\mathbb{S}^{N-1}}}
{\displaystyle\int\limits_{0}^{R}}
\rho^{N-1}V\left(  \rho\right)  \partial_{\rho}\Phi\left(  \rho\right)
\Phi\left(  \rho\right)  \Psi^{2}\left(  \rho u\right)  \partial_{\rho
}J\left(  u,\rho\right)  d\rho du\\
&  =%
{\displaystyle\int\limits_{B_{R}\left(  O\right)  }}
V\left(  \rho\left(  x\right)  \right)  \left\vert \partial_{\rho}f\right\vert
^{2}dV_{g}-%
{\displaystyle\int\limits_{B_{R}\left(  O\right)  }}
V\left(  \rho\left(  x\right)  \right)  \left\vert \partial_{\rho}\left(
\frac{f}{\varphi\left(  \rho\left(  x\right)  \right)  }\right)  \right\vert
^{2}\varphi^{2}\left(  \rho\left(  x\right)  \right)  dV_{g}\\
&  +%
{\displaystyle\int\limits_{B_{R}\left(  O\right)  }}
\left(  \frac{f}{\varphi\left(  \rho\left(  x\right)  \right)  }\right)
^{2}\varphi\left(  \rho\left(  x\right)  \right)  \varphi^{\prime}\left(
\rho\left(  x\right)  \right)  \frac{J^{\prime}\left(  u,\rho\left(  x\right)
\right)  }{J\left(  u,\rho\left(  x\right)  \right)  }dV_{g}.
\end{align*}

\end{proof}

\begin{proof}
[Proof of Theorem \ref{T2}]Let $f\left(  x\right)  -f\left(  \exp_{O}\left(
Ru\right)  \right)  =\varphi\left(  \rho\left(  x\right)  \right)  v\left(
x\right)  $. Then proceed as in the proof of Theorem \ref{T1}, we get
\begin{align*}
&
{\displaystyle\int\limits_{B_{R}\left(  O\right)  }}
V\left(  \rho\left(  x\right)  \right)  \left\vert \nabla_{g}\left(
f-f\left(  \exp_{O}\left(  Ru\right)  \right)  \right)  \right\vert _{g}%
^{2}dV_{g}\\
&  =%
{\displaystyle\int\limits_{B_{R}\left(  O\right)  }}
V\left(  \rho\left(  x\right)  \right)  \left\vert \nabla_{g}\left(
\varphi\left(  \rho\left(  x\right)  \right)  v\left(  x\right)  \right)
\right\vert _{g}^{2}dV_{g}\\
&  =%
{\displaystyle\int\limits_{B_{R}\left(  O\right)  }}
V\left(  \rho\left(  x\right)  \right)  \left\vert \varphi^{2}\left(
\rho\left(  x\right)  \right)  \right\vert \left\vert \nabla_{g}v\right\vert
_{g}^{2}dV_{g}+%
{\displaystyle\int\limits_{B_{R}\left(  O\right)  }}
V\left(  \rho\left(  x\right)  \right)  \left\vert v^{2}\right\vert \left\vert
\nabla_{g}\varphi\left(  \rho\left(  x\right)  \right)  \right\vert _{g}%
^{2}dV_{g}\\
&  +%
{\displaystyle\int\limits_{B_{R}\left(  O\right)  }}
V\left(  \rho\left(  x\right)  \right)  \varphi\left(  \rho\left(  x\right)
\right)  \left\langle \nabla_{g}v^{2},\nabla_{g}\varphi\left(  \rho\left(
x\right)  \right)  \right\rangle _{g}dV_{g}\\
&  =%
{\displaystyle\int\limits_{B_{R}\left(  O\right)  }}
V\left(  \rho\left(  x\right)  \right)  \left\vert \varphi^{2}\left(
\rho\left(  x\right)  \right)  \right\vert \left\vert \nabla_{g}v\right\vert
_{g}^{2}dV_{g}+%
{\displaystyle\int\limits_{B_{R}\left(  O\right)  }}
V\left(  \rho\left(  x\right)  \right)  \left\vert v^{2}\right\vert \left\vert
\varphi^{\prime}\left(  \rho\left(  x\right)  \right)  \right\vert ^{2}%
dV_{g}\\
&  +%
{\displaystyle\int\limits_{B_{R}\left(  O\right)  }}
V\left(  \rho\left(  x\right)  \right)  \varphi\left(  \rho\left(  x\right)
\right)  \varphi^{\prime}\left(  \rho\left(  x\right)  \right)  \left\langle
\nabla_{g}v^{2},\nabla_{g}\rho\left(  x\right)  \right\rangle _{g}dV_{g}.
\end{align*}
Now, using the divergence theorem, we get%
\begin{align*}
&
{\displaystyle\int\limits_{B_{R}\left(  O\right)  }}
V\left(  \rho\left(  x\right)  \right)  \varphi\left(  \rho\left(  x\right)
\right)  \varphi^{\prime}\left(  \rho\left(  x\right)  \right)  \left\langle
\nabla_{g}v^{2},\nabla_{g}\rho\left(  x\right)  \right\rangle _{g}dV_{g}\\
&  =-%
{\displaystyle\int\limits_{B_{R}\left(  O\right)  }}
v^{2}\operatorname{div}\left(  V\left(  \rho\left(  x\right)  \right)
\varphi\left(  \rho\left(  x\right)  \right)  \varphi^{\prime}\left(
\rho\left(  x\right)  \right)  \nabla_{g}\rho\left(  x\right)  \right)
dV_{g}\\
&  -%
{\displaystyle\int\limits_{\partial B_{R}\left(  O\right)  }}
v^{2}V\left(  \rho\left(  x\right)  \right)  \varphi\left(  \rho\left(
x\right)  \right)  \varphi^{\prime}\left(  \rho\left(  x\right)  \right)
\frac{\partial\rho}{\partial\nu}\left(  x\right)  dS_{g}.
\end{align*}
By the assumption (C) on $f$, we get%
\[%
{\displaystyle\int\limits_{\partial B_{R}\left(  O\right)  }}
v^{2}V\left(  \rho\left(  x\right)  \right)  \varphi\left(  \rho\left(
x\right)  \right)  \varphi^{\prime}\left(  \rho\left(  x\right)  \right)
\frac{\partial\rho}{\partial\nu}\left(  x\right)  dS_{g}=0.
\]
Hence%
\begin{align*}
&
{\displaystyle\int\limits_{B_{R}\left(  O\right)  }}
V\left(  \rho\left(  x\right)  \right)  \varphi\left(  \rho\left(  x\right)
\right)  \varphi^{\prime}\left(  \rho\left(  x\right)  \right)  \left\langle
\nabla_{g}v^{2},\nabla_{g}\rho\left(  x\right)  \right\rangle _{g}dV_{g}\\
&  =-%
{\displaystyle\int\limits_{B_{R}\left(  O\right)  }}
v^{2}V\left(  \rho\left(  x\right)  \right)  \varphi\left(  \rho\left(
x\right)  \right)  \varphi^{\prime}\left(  \rho\left(  x\right)  \right)
\Delta_{g}\rho\left(  x\right)  dV_{g}\\
&  -%
{\displaystyle\int\limits_{B_{R}\left(  O\right)  }}
v^{2}V^{\prime}\left(  \rho\left(  x\right)  \right)  \varphi\left(
\rho\left(  x\right)  \right)  \varphi^{\prime}\left(  \rho\left(  x\right)
\right)  dV_{g}\\
&  -%
{\displaystyle\int\limits_{B_{R}\left(  O\right)  }}
v^{2}V\left(  \rho\left(  x\right)  \right)  \varphi^{\prime}\left(
\rho\left(  x\right)  \right)  \varphi^{\prime}\left(  \rho\left(  x\right)
\right)  dV_{g}\\
&  -%
{\displaystyle\int\limits_{B_{R}\left(  O\right)  }}
v^{2}V\left(  \rho\left(  x\right)  \right)  \varphi\left(  \rho\left(
x\right)  \right)  \varphi^{\prime\prime}\left(  \rho\left(  x\right)
\right)  dV_{g}.
\end{align*}
Again, using%
\[
\Delta_{g}\rho\left(  x\right)  =\frac{N-1}{\rho\left(  x\right)  }%
+\frac{J^{\prime}\left(  u,\rho\right)  }{J\left(  u,\rho\right)  }%
\]
we get%
\begin{align*}
&
{\displaystyle\int\limits_{B_{R}\left(  O\right)  }}
V\left(  \rho\left(  x\right)  \right)  \left\vert \nabla_{g}\left(
f-f\left(  \exp_{O}\left(  Ru\right)  \right)  \right)  \right\vert _{g}%
^{2}dV_{g}-%
{\displaystyle\int\limits_{B_{R}\left(  O\right)  }}
V\left(  \rho\left(  x\right)  \right)  \left\vert \varphi^{2}\left(
\rho\left(  x\right)  \right)  \right\vert \left\vert \nabla_{g}v\right\vert
_{g}^{2}dV_{g}\\
&  =-%
{\displaystyle\int\limits_{B_{R}\left(  O\right)  }}
\varphi\left(  \rho\left(  x\right)  \right)  v^{2}\left[
\begin{array}
[c]{c}%
V\left(  \rho\left(  x\right)  \right)  \varphi^{\prime}\left(  \rho\left(
x\right)  \right)  \frac{N-1}{\rho\left(  x\right)  }\\
+V^{\prime}\left(  \rho\left(  x\right)  \right)  \varphi^{\prime}\left(
\rho\left(  x\right)  \right)  +V\left(  \rho\left(  x\right)  \right)
\varphi^{\prime\prime}\left(  \rho\left(  x\right)  \right)
\end{array}
\right]  dV_{g}\\
&  -%
{\displaystyle\int\limits_{B_{R}\left(  O\right)  }}
v^{2}V\left(  \rho\left(  x\right)  \right)  \varphi\left(  \rho\left(
x\right)  \right)  \varphi^{\prime}\left(  \rho\left(  x\right)  \right)
\frac{J^{\prime}\left(  u,\rho\right)  }{J\left(  u,\rho\right)  }dV_{g}\\
&  =%
{\displaystyle\int\limits_{B_{R}\left(  O\right)  }}
\varphi^{2}\left(  \rho\left(  x\right)  \right)  v^{2}W\left(  \rho\left(
x\right)  \right)  -%
{\displaystyle\int\limits_{B_{R}\left(  O\right)  }}
v^{2}V\left(  \rho\left(  x\right)  \right)  \varphi\left(  \rho\left(
x\right)  \right)  \varphi^{\prime}\left(  \rho\left(  x\right)  \right)
\frac{J^{\prime}\left(  u,\rho\right)  }{J\left(  u,\rho\right)  }dV_{g}\\
&  =%
{\displaystyle\int\limits_{B_{R}\left(  O\right)  }}
W\left(  \rho\left(  x\right)  \right)  \left\vert f-f\left(  \exp_{O}\left(
Ru\right)  \right)  \right\vert ^{2}dV_{g}\\
&  -%
{\displaystyle\int\limits_{B_{R}\left(  O\right)  }}
v^{2}V\left(  \rho\left(  x\right)  \right)  \varphi\left(  \rho\left(
x\right)  \right)  \varphi^{\prime}\left(  \rho\left(  x\right)  \right)
\frac{J^{\prime}\left(  u,\rho\right)  }{J\left(  u,\rho\right)  }dV_{g}.
\end{align*}
Similarly,
\begin{align*}
&
{\displaystyle\int\limits_{\mathbb{M}\setminus B_{R}\left(  O\right)  }}
V\left(  \rho\left(  x\right)  \right)  \left\vert \nabla_{g}\left(
f-f\left(  \exp_{O}\left(  Ru\right)  \right)  \right)  \right\vert _{g}%
^{2}dV_{g}\\
&  -%
{\displaystyle\int\limits_{\mathbb{M}\setminus B_{R}\left(  O\right)  }}
V\left(  \rho\left(  x\right)  \right)  \left\vert \varphi^{2}\left(
\rho\left(  x\right)  \right)  \right\vert \left\vert \nabla_{g}\left(
\frac{f-f\left(  \exp_{O}\left(  Ru\right)  \right)  }{\varphi\left(
\rho\left(  x\right)  \right)  }\right)  \right\vert _{g}^{2}dV_{g}\\
&  =%
{\displaystyle\int\limits_{\mathbb{M}\setminus B_{R}\left(  O\right)  }}
W\left(  \rho\left(  x\right)  \right)  \left\vert f-f\left(  \exp_{O}\left(
Ru\right)  \right)  \right\vert ^{2}dV_{g}\\
&  -%
{\displaystyle\int\limits_{\mathbb{M}\setminus B_{R}\left(  O\right)  }}
v^{2}V\left(  \rho\left(  x\right)  \right)  \varphi\left(  \rho\left(
x\right)  \right)  \varphi^{\prime}\left(  \rho\left(  x\right)  \right)
\frac{J^{\prime}\left(  u,\rho\right)  }{J\left(  u,\rho\right)  }dV_{g}.
\end{align*}
Therefore%
\begin{align*}
&
{\displaystyle\int\limits_{\mathbb{M}}}
V\left(  \rho\left(  x\right)  \right)  \left\vert \nabla_{g}\left(
f-f\left(  \exp_{O}\left(  Ru\right)  \right)  \right)  \right\vert _{g}%
^{2}dx-%
{\displaystyle\int\limits_{\mathbb{M}}}
W\left(  \rho\left(  x\right)  \right)  \left\vert f-f\left(  \exp_{O}\left(
Ru\right)  \right)  \right\vert ^{2}dV_{g}\\
&  =%
{\displaystyle\int\limits_{\mathbb{M}}}
V\left(  \rho\left(  x\right)  \right)  \varphi^{2}\left(  \rho\left(
x\right)  \right)  \left\vert \nabla_{g}\left(  \frac{f-f\left(  \exp
_{O}\left(  Ru\right)  \right)  }{\varphi\left(  \rho\left(  x\right)
\right)  }\right)  \right\vert _{g}^{2}dx\\
&  -%
{\displaystyle\int\limits_{\mathbb{M}}}
V\left(  \rho\left(  x\right)  \right)  \left\vert f-f\left(  \exp_{O}\left(
Ru\right)  \right)  \right\vert ^{2}\frac{\varphi^{\prime}\left(  \rho\left(
x\right)  \right)  }{\varphi\left(  \rho\left(  x\right)  \right)  }%
\frac{J^{\prime}\left(  u,\rho\right)  }{J\left(  u,\rho\right)  }dV_{g}.
\end{align*}
Similarly,%
\begin{align*}
&
{\displaystyle\int\limits_{\mathbb{M}}}
V\left(  \rho\left(  x\right)  \right)  \left\vert \partial_{\rho}\left(
f-f\left(  \exp_{O}\left(  Ru\right)  \right)  \right)  \right\vert _{g}%
^{2}dx-%
{\displaystyle\int\limits_{\mathbb{M}}}
W\left(  \rho\left(  x\right)  \right)  \left\vert f-f\left(  \exp_{O}\left(
Ru\right)  \right)  \right\vert ^{2}dV_{g}\\
&  =%
{\displaystyle\int\limits_{\mathbb{M}}}
V\left(  \rho\left(  x\right)  \right)  \varphi^{2}\left(  \rho\left(
x\right)  \right)  \left\vert \partial_{\rho}\left(  \frac{f-f\left(  \exp
_{O}\left(  Ru\right)  \right)  }{\varphi\left(  \rho\left(  x\right)
\right)  }\right)  \right\vert ^{2}dx\\
&  -%
{\displaystyle\int\limits_{\mathbb{M}}}
V\left(  \rho\left(  x\right)  \right)  \left\vert f-f\left(  \exp_{O}\left(
Ru\right)  \right)  \right\vert ^{2}\frac{\varphi^{\prime}\left(  \rho\left(
x\right)  \right)  }{\varphi\left(  \rho\left(  x\right)  \right)  }%
\frac{J^{\prime}\left(  u,\rho\right)  }{J\left(  u,\rho\right)  }dV_{g}.
\end{align*}

\end{proof}

\end{document}